\pdfoutput=1
\RequirePackage{ifpdf}
\ifpdf 
\documentclass[pdftex]{sigma}
\else
\documentclass{sigma}
\fi

\usepackage[v2,all]{xy}
\usepackage{mathtools,mathrsfs}
\usepackage{tikz-cd}
\usepackage{mathabx}

\numberwithin{equation}{section}

\newtheorem{Theorem}{Theorem}[section]
\newtheorem*{Theorem*}{Theorem}
\newtheorem*{TheoremMain1}{Theorem~\ref{mainthm}}
\newtheorem*{TheoremMain2}{Theorem~\ref{mainthm2}}
\newtheorem*{TheoremMain3}{Theorem~\ref{thm:mainpathlike}}
\newtheorem{Corollary}[Theorem]{Corollary}
\newtheorem{Lemma}[Theorem]{Lemma}
\newtheorem{Proposition}[Theorem]{Proposition}
 { \theoremstyle{definition}
\newtheorem{Definition}[Theorem]{Definition}

\newtheorem{Example}[Theorem]{Example}
\newtheorem{Remark}[Theorem]{Remark}
\newtheorem{Notation}[Theorem]{Notation} }

\newcommand{\C}{\catname{C}}
\newcommand{\F}{\catname{F}}

\newcommand{\im}{{\mathrm {im}}\,}
\newcommand{\Id}{{\rm id}}

\def\FF{{\mathbb F}}
\def\GG{{\mathbb G}}

\def\text#1{{\rm {\rm #1}}}

\def\max#1{|#1|_{\rm max}}
\def\Set{\dul{\rm Set}}

\def\bq{\mathbf {q}}

\usetikzlibrary{cd,fit}

\usepackage{bold}
 \def\unit{\Eins}
 \def\gh{\mathbbnew{\Gamma}}

\def\C{\mathcal C}

\def\O{\mathcal{O}}

\def\Z{\mathbb{Z}}

\def\Set{\mathcal{S}et}

\def\V{\mathcal V}

\def\E{\mathcal{E}}
\newcommand{\leftsub}[2]{{\vphantom{#2}}_{#1}{#2}}
\newcommand{\bisub}[3]{{\vphantom{#2}}_{#1}{#2}_{#3}}
\def\stimest{\bisub{s}{\times}{t}}

\makeatother
\newcommand{\EZDIAG}[5]{\xymatrix
@C+=2.5cm{*+[r]{#1}
\ar@(u,l)_(0.62){\displaystyle #5}[]
\ar@<1ex>^-{#3}[r]&\ar@<1ex>^-{#4}[l]#2}}

\def\id{{\mathrm{id}}}

\def\SS{\mathbb{S}}
\def\odo{\otimes \dots\otimes}

\def\kdk{,\dots,}

\def\eps{\epsilon}

\def\B{\mathscr{B}}

\def\C{\mathscr {C}}

\def\F{\mathcal{F}}
\def\FF{\mathfrak F}
\def\asts{{\mathcal V}}

\def\clusters{{\mathcal F}}
\def\N{\mathbb{N}}

\def\eps{\epsilon}

\def\CalC{{\mathcal C}}

\def\Set{{\mathcal S}et}

\def\F{\mathcal F}
\def\FF{\mathfrak F}

\def\GG{\mathfrak G}

\def\C{\CalC}
\def\Z{{\mathbb Z}}
\def\N{{\mathbb N}}

\def\FinSet{\mathcal{F}{\rm in}\mathcal{S}{\rm et}}

\def\s{\sigma}

\def\t{\tau}
\def\g{\gamma}

\def\O{{\mathcal O}}

\def\SS{{\mathbb S}}

\def\odo{\otimes \cdots \otimes}

\def\deg{\mathrm{deg}}

\newcommand\ccirc[2]{\, \leftsub{#1}{\circ}_{#2}}
\def\scirct{\ccirc{s}{t}}

\newcommand\mge[2]{\, \leftsub{#1}{\boxminus}_{#2}}

\newcommand{\Obj}{\operatorname{Obj}}
\newcommand{\Hom}{\operatorname{Hom}}
\newcommand{\Iso}{\operatorname{Iso}}

\newcommand{\Mor}{\operatorname{Mor}}

\def\V{\asts}

\def\asts{{\mathcal V}}
\def\F{\clusters}
\def\clusters{{\mathcal F}}

\def\D{\mathcal D}

\def\d{\delta}

\def\Biso{B^{\rm iso}}

\def\Diso{\Delta^{\rm iso}}

\def\sds{(\sigma \Da \sigma')}

\def\ot{\otimes}

\def\Da{\!\Downarrow\!}

\tikzstyle{none}=[]
\tikzstyle{Box}=[fill=white, draw=black, shape=rectangle]
\tikzstyle{Dotted}=[fill=white, draw=black, shape=rectangle, dashed]
\tikzstyle{Gap}=[fill=white, shape=circle]
\tikzstyle{White}=[fill=white, draw=black, shape=circle]
\tikzstyle{Black}=[fill=black, draw=black, shape=circle]
\tikzstyle{Gray}=[fill={rgb,255: red,191; green,191; blue,191}, draw=black, shape=circle]
\tikzstyle{myRed}=[fill=red, draw=black, shape=circle]
\tikzstyle{myBlue}=[fill=blue, draw=black, shape=circle]
\tikzstyle{myGreen}=[fill=green, draw=black, shape=circle]
\tikzstyle{Cyan}=[fill=cyan, draw=black, shape=circle]
\tikzstyle{Yellow}=[fill=yellow, draw=black, shape=circle]
\tikzstyle{Magenta}=[fill=magenta, draw=black, shape=circle]

\tikzstyle{Arrow}=[->]
\tikzstyle{Mapsto}=[{|->}]
\tikzstyle{Dash}=[-, dashed]
\tikzstyle{DashArrow}=[->, dashed]
\tikzstyle{Double}=[<->]
\tikzstyle{Red}=[-, draw=red]
\tikzstyle{Blue}=[-, draw=blue]
\tikzstyle{Green}=[-, draw=green]
\tikzstyle{RedArrow}=[draw=red, ->]
\tikzstyle{BlueArrow}=[draw=blue, ->]

\def\SGL{\operatorname{sgpl}}
\def\GLE{\operatorname{gpl}}
\def\D{\Delta}
\def\bD{\overline{\Delta}}
\def\bDn{\overline{\Delta}^{[n]}}

\def\bot{\boxtimes}
\def\ve{\vec{e}}

\newcommand{\Free}[1]{R^{(#1)}}
\newcommand{\finite}[1]{\check #1^{\circ}}

\newcommand{\elt}[2]{\, #1 \underset{\displaystyle #2}{\boxvoid}}

\def\gx{\elt{g}{x}}
\def\hy{\elt{h}{y}}
\def\kz{\elt{k}{z}}

\def\dkg{D(k[G])}
\def\Sp{^{(1)}}
\def\Spp{^{(2)}}
\def\Sppp{^{(3)}}

\def\FQT{F^{QT}}
\def\FBQ{F^{Q_{\bullet\bullet}}}
\def\FN{F^N}
\def\FQ{F^Q}

\begin{document}
\allowdisplaybreaks

\newcommand{\arXivNumber}{2104.08895}

\renewcommand{\thefootnote}{}

\renewcommand{\PaperNumber}{053}

\FirstPageHeading

\ShortArticleName{Pathlike Co/Bialgebras and their Antipodes}

\ArticleName{Pathlike Co/Bialgebras and their Antipodes\\ with Applications to Bi- and Hopf Algebras\\ Appearing in Topology, Number Theory and Physics\footnote{This paper is a~contribution to the Special Issue on Algebraic Structures in Perturbative Quantum Field Theory in honor of Dirk Kreimer for his 60th birthday. The~full collection is available at \href{https://www.emis.de/journals/SIGMA/Kreimer.html}{https://www.emis.de/journals/SIGMA/Kreimer.html}}}

\Author{Ralph M.~KAUFMANN~$^{\rm ab}$ and Yang MO~$^{\rm a}$}

\AuthorNameForHeading{R.M.~Kaufmann and Y.~Mo}

\Address{$^{\rm a)}$~Department of Mathematics, Purdue University, West Lafayette, IN, USA}
\EmailD{\href{mailto:rkaufman@purdue.edu}{rkaufman@purdue.edu}, \href{mailto:mo12@purdue.edu}{mo12@purdue.edu}}
\URLaddressD{\url{https://www.math.purdue.edu/~rkaufman/}}

\Address{$^{\rm b)}$~Department of Physics and Astronomy, Purdue University, West Lafayette, IN, USA}

\ArticleDates{Received April 18, 2021, in final form June 29, 2022; Published online July 11, 2022}

\Abstract{We develop an algebraic theory of colored, semigrouplike-flavored and pathlike co-, bi- and Hopf algebras. This is the right framework in which to discuss antipodes for bialgebras naturally appearing in combinatorics, topology, number theory and physics. In~particular, we can precisely give conditions for the invertibility of characters that is needed for renormalization in the formulation of Connes and Kreimer. These are met in the relevant examples. In~order to construct antipodes, we discuss formal localization constructions and quantum deformations. These allow to define and explain the appearance of Brown style coactions. Using previous results, we can interpret all the relevant coalgebras as stemming from a categorical construction, tie the bialgebra structures to Feynman categories, and apply the developed theory in this setting.}

\Keywords{Feynman category; bialgebra; Hopf algebra; antipodes; renomalization; characters; combinatorial coalgebra; graphs; trees; Rota--Baxter; colored structures}

\Classification{16T05; 18M85; 81T15; 81R50}

\begin{flushright}
\begin{minipage}{63mm}
 \it Dedicated to Professor Dirk Kreimer\\ on the occasion of his $60$th birthday
\end{minipage}
\end{flushright}

\renewcommand{\thefootnote}{\arabic{footnote}}
\setcounter{footnote}{0}

\section{Introduction}
Although the use of Hopf algebras has a long history, the seminal paper \cite{CK} led to a turbocharged development for their use which has penetrated deeply into mathematical physics, number theory and also topology, their original realm~-- see \cite{HopfHistory} for the early history. The important realization in \cite{CK,CK1,CK2} was that the renormalization procedure in quantum field theory can be based on a character theory for Hopf algebras via the so-called Birkhoff factorization and convolution inverses.
The relevant Hopf algebras are those of trees with a coproduct given by a sum over so-called admissible cuts~-- with the factors of the coproduct being the left-over stump and the collection of cut-off branches~-- and Hopf algebras of graphs in which the factors of the summands of the coproduct are given by a subgraph and the quotient graph. A planar version of the tree formalism was pioneered in \cite{foissyCR1}, see also \cite{MKW}. The appearance in number theory of this type of coproduct goes back to \cite{Gont}. It was developed further and applied with great success in \cite{BrownICM}. In~all these cases, the product structure is free and the coproduct is the carrier of information. The group of characters was previously studied in \cite{Butcher}.
A precursor to the Hopf algebraic considerations can be found in \cite{baues}. In~\cite{HopfPart1,HopfPart2}, we gave the details to prove the results announced in \cite{feynman} that all these structures stem from a categorical framework, where the coproduct corresponds to deconcatenation of morphisms. Such coproduct structures can be traced back to \cite{Moebiusguy} and appear in combinatorics \cite{JR,Schmitt}. Dual algebra structures can be found in \cite{Rota} and \cite{Duer}.\footnote{We thank a referee for pointing out this reference.} In~\cite{feynman}, we developed a theory of so-called Feynman categories, which are essentially a special type of monoidal category and could connect the product structure to the monoidal structure on the level of morphisms. Special cases are related to operads, the simplicial category and categories whose morphisms are classified by graphs. In~\cite{HopfPart2}, we could show that the monoidal product in these categories is compatible with the deconcatenation co-product thus yielding bialgebras. This generalizes those of Baues and Goncharov \cite{baues,Gont} which are simplicial in nature,
 the Connes--Kreimer tree Hopf algebras, which have operadic origin~\cite{HopfPart1}, and the Connes--Kreimer Hopf algebras of graphs as well as the core Hopf algebra \cite{Kreimercore}, which are graphical and more categorical in nature.

The co- and bialgebras which have a categorical interpretation also include path coalgebras and incidence coalgebras, as for instance considered in \cite{JR,Schmitt}, see Section~\ref{sec:keyex}. There are two versions of the story, one is symmetric and yields cocommutative bialgebras and the second is non-symmetric and in general yields to non-cocommutative bialgebras, such as those from planar structures. In~all the examples one passes from the bialgebras to a connected quotient to obtain an antipode by invoking Quillen's formalism~\cite{Quillen}.

In this paper, we address the question of obstructions to constructing antipodes on the bialgebra level.
The questions that we can now answer is:
\begin{quote}
What structures precisely exist already on the bialgebra level, and what is their role in the inversion of characters and the construction of antipodes?
\end{quote}
To this end, we give an {\em algebraic characterization} of a classes of bialgebras amenable to such considerations. These are {\em colored}, {\em sg-flavored} (sg stands for semigrouplike), and {\em pathlike} coalgebras. They capture and generalize the essential features of categorical coalgebras, whose paradigmatic example are path algebras for quivers. This allows us to specify effective criteria for the
existence of antipodes and convolution invertibility of characters. We also work over an arbitrary commutative ground ring. The origins for the theory of sg-flavored coalgebras can be traced back to \cite{TaftWilson} and with hindsight, several of the quotient construction are foreshadowed in \cite{Schmitt} for the special case of incidence bialgebras.\footnote{We thank D.~Kreimer for pointing this out.}
The existence of antipodes is usually often established via some sort of connectivity of a filtration \cite{Quillen,Takeuchi}. The problem is that the bialgebras in question are usually not connected for the standard filtrations, which is why quotients are taken.

Concretely, after reviewing some basic structures in the generality we need in Section~\ref{basicpar}, providing a short list of paradigmatic examples to illustrate and motivate the constructions Section~\ref{sec:key}, and touching upon the complications from isomorphisms, we introduce the first key notion, that of colored coalgebras in Section~\ref{sec:colored}. This allows us to generalize Quillen's connected coalgebras and conilpotent coalgebras to many colors. In~the categorical setting, the colors are the objects and algebraically speaking they correspond to grouplike elements. The main structural result is Theorem~\ref{thm:colorstructure}.
The basic examples like path algebras are of this type, and appear as the dual coalgebras $C[M]$ to colored monoids $M$.
Proposition~\ref{prop:catcolor} classifies their properties and gives criteria for color connectedness. Dualities in a general setting are tricky; we provide some technical details and the dual notion of colored algebras in Appendix~\ref{sec:duals}.

To construct antipodes and convolution inverses, we introduce the fundamental concept of a~QT-filtration (Quillen--Takeuchi) in Section~\ref{takeuchipar}, which generalizes the commonly used filtrations. Such a filtration is called a sequence if it is exhaustive. This is what is needed to make
Takeuchi's argument work and results in:

\begin{TheoremMain1} 
Consider a~coalgebra $C$ with counit $\eps$ and an algebra $A$ with unit $\eta$.
	Given a~QT-sequence $\{F_i\}^{\infty}_{i=0}$ on $C$, such that ${\rm Ext}^1_R(C/F_0C,A)=0$,
	for any element $f$ of $\Hom(C, A)$, $f$~is $\star$-invertible if and only if
	the restriction $f|_{F_{0}C}$ is $\star$-invertible in $\Hom(F_{0}C, A)$.
\end{TheoremMain1}
	
A special type of QT-filtration, which we call the bivariate Quillen filtration (see Section~\ref{sg-flavoredpar}), arises by considering skew primitive elements. These are the essential actors for so-called sg-flavored coalgebras and are key to understanding the structures of the bialgebras before taking quotients. In~the applications a slightly stronger condition is satisfied for the level $0$ part of the filtration, which leads to the central notion of pathlike co/bialgebras, cf.\ Section~\ref{sec:pathlike}.	

\begin{TheoremMain2} 
Let $C$ be pathlike coalgebra and $A$ algebra with $\mathop{\rm Ext}\big(C/\FBQ_0C,A\big)=0$. An element $f \in \Hom(C,A)$ has an $\star$-inverse in the convolution algebra $\Hom(C,A)$ if and only if for every grouplike element $g$, $f(g)$ has an inverse as ring element in $A$.

In particular, in this situation, a character is $\star$-invertible if and only if it is grouplike inver\-tible.
	
Furthermore, a pathlike bialgebra $B$ is a Hopf algebra if and only if the set of grouplike elements form a group.
\end{TheoremMain2}

Color connected coalgebras are pathlike and hence these general results apply to them.
	
These results are then directly applicable to renormalization via Rota--Baxter algebras and Birkhoff decompositions, which are briefly reviewed in Section~\ref{sec:renom}, with more details in Appendix~\ref{rbapp}. A closer inspection of
this framework yields two paths of actions. The first is to realize that for invertibility of specific characaters the bialgebras need not necessarily be Hopf algebras.
Thus limiting the characters by imposing certain restrictions on them or the target algebra will ensure their invertibility in the pathlike case. The second avenue is to formally invert the grouplike elements. Combining the two approaches yields several universal constructions, through which characters with special properties factor, cf.\ Section~\ref{renompar}. This for instance
naturally leads to {\em quantum deformations} of the algebras. These results generalize and explain
similar constructions in~\cite{HopfPart1}.
They allow us to construct
 Brown style coactions \cite{BrownICM}, see Proposition \ref{prop:bcoaction}.
 Applying this to pathlike bialgebras, viz.\ bialgebras whose coalgebras are pathlike yields:
\begin{TheoremMain3} 
Let $B$ be a pathlike bialgebra:
Every grouplike invertible character has a $\star$-inverse and vice-versa.

The quotient bialgebra $B/I_N$ of Proposition~{\rm \ref{normalizedprop}} is connected and hence Hopf.

In particular, every grouplike normalized character $\phi\in \Hom(B,A)$, when factored through to~$\bar\phi\in \Hom(B/I_N,A)$, has an inverse computed by $\bar \phi^{-1}=\bar\phi\circ S$.

The bialgebra $B_\bq/I$ is Hopf and the $\star$ inverse of a grouplike central character $\phi\in \Hom(B,A)$, when lifted and factored through $\overline{\hat\phi}\in \Hom(B_\bq/I,A)$, has an inverse computed by $\overline{\hat\phi}^{-1}=\overline{\hat\phi}\circ S$.

Let $S=\GLE(B)$, viz.\ the grouplike elements, then the left localization $BS^{-1}$ is a Hopf algebra. Moreover if $S$ is satisfies the Ore condition and is cancellable, there is an injection $B\to BS^{-1}$ and $B_q/I=BS^{-1}/J$, where $J$ is the ideal generated by $(bg-gb)$ for all $b\in B$, $g\in \GLE(B)$.
\end{TheoremMain3}

This framework is applicable to bialgebras arising from Feynman categories, as we show in Section~\ref{catpar}, resulting in the main structural Theorem \ref{thm:mainns} in the non-symmetric case and Theorem~\ref{thm:isofey} in the symmetric case.
We concretely apply this to the classes of Feynman categories, set based or simplicial, (co)operadic and graphical corresponding to the examples mentioned in the beginning. To be more self-contained, these are given in a new set-theoretical presentation in an Appendix \ref{graphapp}.

We end in Section~\ref{sec:conclusion} with a conclusion and an outlook, which in particular discusses strategies going beyond the pathlike setting by using double categories and Drinfel'd doubles, cf.~Section~\ref{sec:drin}.

\section{Colored algebras and coalgebras}
Before delving into the constructions, we review the general notions, recall the key examples and introduce the notion of colored coalgebras, which formalizes the type of coalgebra obtained from a category or equivalently a partial monoid that is colored.
We end with a short overview of the complications that arise in the presence of isomorphisms. One resolution to this problem is to work with equivalence
classes, which is done in Section~\ref{catpar}.

\subsection{Setup and basic notation}\label{basicpar}

We provide some of the relevant notions and refer to \cite{Cartier,Susan} for more details.
Throughout we work over a commutative unital ground ring $R$. If this is taken to be a field, then it is denoted by $k$. We will denote the free
$R$-module on a set $X$ by $\Free{X}=\bigoplus_{x\in X}R$. Tensors are understood to be over the ground ring, thus $\otimes$ means $\otimes_R$.
For an $R$-module $M$ there is a~canonical isomorphism $M \otimes R\simeq M$ given by the $R$-module structure on $M$, and we will simply implement this isomorphism as an identification-effectively making $R$ a strict monoidal unit.
We will either work in the category of $R$-modules {\em or graded} $R$-modules. In~the second case, we assume that $R$ is concentrated in degree $0$ and the grading is by $\mathbb{N}$ unless otherwise stated.
We~will assume that algebras have a unit and coalgebras have a counit unless otherwise stated.
A~graded coalgebra is connected if $C_0=R$.

\subsubsection{Internal product}

There are several complications when working over a commutative ring. First,
there is no canonical way to identify a tensor product of submodules with a submodule.
This may lead to different subcoalgebra structures.
Second, a submodule with the restriction of the coalgebra structure may not yield a coalgebra.
It may not even be possible to give a coalgebra structure.
To avoid these issues, we will use an internal product denoted by $\boxtimes$ for the underlying $R$-submodules.

Given a coalgebra $(C, \Delta, \eta)$ over $R$ and two submodules $M$, $N$ of $C$, their internal product $M \boxtimes N$ is defined to be the following submodule of $C \otimes C$:
\begin{gather*}
M \boxtimes N:=\{\text{finite~sums~of~}\, x_{(1)} \otimes x_{(2)} \in C \otimes C \,|\, x_{(1)} \in M,\, x_{(2)} \in N\}.
\end{gather*}
For two families of subcomodules $\{N_j\}$, $\{M_i\}$, we will also use the notation $\Delta(x) \in \sum_{i,j} M_i \boxtimes N_j$. This means $\Delta(x)$ can be decomposed as a finite sum of
$x_{(1)} \otimes x_{(2)}$ such that every summand belongs to $M_i \boxtimes N_j$ for some $i$, $j$.
If $M \subset M'$ and $N \subset N'$ as submodules, then $M \boxtimes N \subset M' \boxtimes N'$.
For a submodule $S$, we say that $V$ {\em is a subcoalgebra}, if $\Delta(S) \subset S \boxtimes S$. Note that the counit $\eps|_S\colon S\to R$ satisfies the equations of a counit.

\begin{Example}[\cite{NicholsSweedler}]
Consider the $\mathbb{Z}$-module
$C =\frac{\mathbb{Z}}{8\mathbb{Z}} \oplus \frac{\mathbb{Z}}{2\mathbb{Z}}$. Let $x = (1,0)$
and $z = (0,1)$ and endow $C$ with a coalgebra structure by setting $\Delta(x) = 0$
and $\Delta(z) = 4x \otimes x$ . Since $4x \otimes x$ has order of 2 in $C \otimes C$ the coproduct is well-defined.
Also, $\epsilon(x) = \epsilon(z) = 0$, because they are torsion elements. Let $y = 2x$ and consider $V = \mathbb{Z}y + \mathbb{Z}z \subset C$. As $\Delta(z) = y \otimes y$, $\Delta(V) \subset V \boxtimes V$. But, there is no coalgebra structure on $V$ such that the natural inclusion is a coalgebra morphism. Indeed, the map $\Delta\colon V \to C \otimes C$ has no lifting to $V \otimes V$ because every preimage of $\Delta(z)$ has order 4.
\end{Example}

The complication arises due to torsion elements, in the case that $R=k$ is a field, we can work with the usual monoidal product $\otimes$.

\subsubsection{Semigrouplikes and connected coalgebras }
Given a coalgebra $C$, an element $g$ in $C$ is said to be \textit{semigrouplike} if $\Delta(g) = g \otimes g$. If additionally
$\epsilon(g) = 1$ then $g$ is said to be \textit{grouplike}.
Note that $0$ is always semi-grouplike, but never grouplike.
	The set of {\em non-zero} semigrouplike elements in $C$ is denoted by $\SGL(C)$ and the set
of grouplike elements in $C$ by $\GLE(C)$.
If the coalgebra has a counit $\eps$, then for $g\in \SGL(G)$: $(1-\eps(g))g=0$ that is $1-\eps(g)\in {\rm Ann}_R(g)$. Hence
such a $g$ is a torsion element, if it is not grouplike.
		Let $g$ and $h$ be grouplike elements in $C$, an element $x$ in $C$ is called $(g,h)$-\textit{skew primitive}
if $\Delta(x) = x \otimes g + h \otimes x$. 	When $g = h$, it is called $g$-\textit{primitive element}.

In a coaugmented coalgebra, the coaugmentation $\eta\colon R\to C$ preserves the counit.
Thus, setting $e=\eta(1)$, $\Delta(e)=e \ot e$, there is a splitting $C={\rm Re}\oplus \overline{C}$ and $\overline{C}=\ker(\epsilon)$.
The reduced diagonal $\bD$ is defined by $\bD(x)=\D(x)-e\ot x-x\ot e$. In~a graded coalgebra all grouplike elements are necessarily in degree $0$ as $2i=i$ implies $i=0$. In~a connected graded coalgebra therefore
$e$, which is the generator of $C_0$, is the unique grouplike element.

In a bialgebra, the semigrouplike elements form a monoid: for two semigrouplike elements $g,h\colon \Delta(gh)=gh\ot gh$.
The grouplike elements form a submonoid:
$\eps(gh)=\eps(g)\eps(h)=1$. If $C$ is coaugmented, these structures are unital with unit $e$.

In a Hopf algebra any semigrouplike element $g$ satisfies $S(g)g=gS(g)=\eps(g)e$. In~particular, if an element is grouplike, $S(g)=g^{-1}$ is the
only possible value for an antipode.

\subsubsection{Convolution algebra and characters}

Recall that given a coalgebra $C$ and an algebra $A$, the {\em convolution algebra} is the $R$-module
$\Hom(C, A)$ with multiplication given by $(f \star g) (x) = m \circ (f \otimes g) \circ \Delta (x) = \sum_{x} f(x_{(1)}) g(x_{(2)})$.
If~$A$ is unital with unit $\eta$ and $C$ is counital with counit $\eps$ then $\eta \circ \epsilon$ is a unit for the convolution algebra.

\begin{Lemma}
\label{restrictionlem}
 Given a subcoalgebra $S$,
$\Hom(S,A)$ is a unital subalgebra of the convolution algebra with unit $(\eta\otimes\eps)|_{S}=\eta\circ \eps|_{S}$.
\end{Lemma}

\begin{proof}
One has to check that $\Hom(S,C)$ is closed under composition and that it contains the unit.
Since $\Delta(S)\subset S\boxtimes S$, the formula $(f\star g)(s)=m(f\ot g)\circ \Delta(s)=\sum f(s_{(1)})g(s_{(2)})$
yields a~function from $S\boxtimes S\to C$. The unit restricts appropriately.
\end{proof}

\begin{Lemma} 
	If $C$ is coaugmented and $A$ is augmented, then if $f$ preserves the augmentation or the coaugmentation, so does its convolution inverse.
\end{Lemma}

\begin{proof} Assume $f(1_C) = 1_A$, we need to show $g(1_C) = 1_A$.
	On one hand
\begin{gather*}	
(f \star g) (1_C)= m_A \circ (f \otimes g) \circ \Delta_C (1_C)
= (m_A \circ (f \otimes g))(1_C \otimes 1_C)
=f(1_C)g(1_C) = g(1_C).
\end{gather*}
	On the other hand, $(f \star g) (1_C)= \eta_A\eps_C(1_C)=1_A$.
The compatibility with the counit is established by:
\begin{align}	
\epsilon_C(a) = {}&\epsilon_C(a)1_R = \epsilon_C(a)\epsilon_A(1_A)=\epsilon_A( \epsilon_C(a)1_A) =\epsilon_A( \epsilon_C(a)\eta_A(1_A))=\epsilon_A( \eta_A(\epsilon_C(a)))\nonumber
\\
 ={}& \epsilon_A\bigg(\sum_af(a_{(1)})g(a_{(2)})\bigg)
= \sum_a\epsilon_A(f(a_{(1)}))\epsilon_A(g(a_{(2)}))
=\sum_a\epsilon_C(a_{(1)})\epsilon_A(g(a_{(2)}))\nonumber
\\
&\times\epsilon_A\bigg(\sum_a\epsilon_C(a_{(1)})g(a_{(2)})\bigg) = \epsilon_A\bigg(g\bigg(\sum_a\epsilon_C(a_{(1)})a_{(2)}\bigg)\bigg)
= \epsilon_A(g(a)).
\tag*{\qed}
\end{align}
\renewcommand{\qed}{}
\end{proof}

An important fact is that an antipode $S$ for a Hopf algebra $H$ is a convolution inverse to $\id\in \Hom(H,H)$, where on the left the coalgebra structure
of $H$ is used and on the right, the algebra structure of $H$ is used. As $\id$ preserves the coaugmentation and augmentation, so does~$S$.

If $B$ is a bialgebra over $R$ and $A$ is algebra over $R$, then the set of {\em characters} of $B$ with values in $A$ are the algebra homorphisms
$\Hom_{R\text{-}{\rm alg}}(B,A)$.

\begin{Definition}
We say a character is {\em grouplike invertible}, if for every grouplike element $g\colon \phi(g)\in A^\times$, {\em grouplike central} if for all $g\in \GLE(C)\colon \phi(g)\in Z(A)$, grouplike
scalar if $\phi(g)\in R^\times$ and grouplike normalized if $\phi(g)=1$ for all grouplike $g$.
\end{Definition}

Recall that in a graded setting, the grouplike elements are in degree $0$, and if $A$ is graded and~$\phi$ preserves the grading, then the grouplike elements have to land in $A_0$.
It is quite common that, even if $A$ is not commutative, $A_0$ is. In~this case all the characters preserving the grading will be grouplike central.
 In several applications, the characters are scalar and
take values $\big(\frac{1}{2\pi {\rm i}}\big)^k$, which in turn is also a form of grading, cf.\ Section~\ref{sec:quot}.

\begin{Lemma}
Any $\star$-invertible character is grouplike invertible.
\end{Lemma}

\begin{proof}
This follows from the fact that for any grouplike
 $\phi^{\star -1}\star \phi(g)= \phi^{\star -1}(g)\phi(g)=1_B$.
 \end{proof}

The converse of this is true under a specific conditions on $A$ for pathlike bialgebras, see Theorem \ref{mainthm2} and Section~\ref{renompar}.

\begin{Proposition}
\label{convolutionprop}

 If $A$ is commutative, the characters form an algebra under convolution.

If $B$ is a Hopf algebra, then $\phi\circ S=\phi^{-1}$ is the {\em convolution inverse}.

If $B$ is Hopf and $A$ is commutative, then the characters form a group.
\end{Proposition}

\begin{proof}
For the first statement:
\begin{align*}
(f\star g)(ab)&=(f\ot g)\bigg(\sum a_{(1)}b_{(1)}\ot a_{(2)} b_{(2)}\bigg)=\sum f(a_{(1)})f(b_{(1)})g(a_{(2)})g(b_{(2)})
\\
&=\sum f(a_{(1)})g(a_{(2)})f(b_{(1)})g(b_{(2)})=(f\star g)(a)(f\star g)(b).
\end{align*}

For the second statement:
\begin{gather*}
((\phi\circ S)\star \phi)(a)=\!\sum_a\phi(S(a_{(1)})) \phi(a_{(2)})=\!\sum_a \phi(a_{(1)}S(a_{(2)}))=\phi(\eta\circ \eps(a))
= (\eta\circ\eps)(a).
\end{gather*}

For the last statement, we need to show that the inverse of a character is a character. Indeed, $\phi^{-1}(ab)=\phi(S(ab))=\phi(S(b)S(a))=\phi(S(b)\phi(S(a))=\phi(S(a))\phi(S(b))$, where the last equation holds since $A$ is commutative.
\end{proof}

\subsection{Key examples of coalgebras and complications}
\label{sec:key}
\subsubsection{Path coalgebra}
\label{sec:keyex}
The paradigmatic example and namesake for the article is the {\em path coalgebra} of a quiver.
Given a quiver $Q$, that is a graph with directed edges, a path is given by a sequence of consecutive directed edges $p=(\ve_1\cdots\ve_n)$.
Consecutive means that the target vertex $t(\ve_i)$ of $\ve_i$ is the source vertex $s(\ve_{i+1})$ of $\ve_{i+1}$.
By definition for each vertex $v$ there is an empty path of length $0$, with source and target $v$. This
the identity at $v$, which is traditionally denoted simply by $v$. Using this notation, let $P(Q)$ be the set of paths of $Q$, then
 $\Free{P(Q)}$ has the coalgebra structure
\begin{gather*}
\Delta(\ve_1\cdots \ve_n)=
{s(\ve_1)}\ot (\ve_1\cdots \ve_n) +
\sum_{i=1}^{n-1} (\ve_1\cdots \ve_i)\ot (\ve_{i+1}\cdots \ve_n)+ (\ve_1\cdots \ve_n)\ot {t(\ve_n)}.
\end{gather*}
The counit is $\eps(v)=1$ and $\eps(\ve_1\cdots \ve_n)=0, n\geq 0$.
There is a grading given by the length of a~path.
The grouplike elements are exactly the length $0$ paths $v$. The paths of length $1$ are exactly the skew primitive, with $(\ve)$ being
$(s(\ve),t(\ve))$-skew primitive. This coalgebra is not connected if there is more than one vertex.

\begin{Remark}
A special important case arises if one considers the quiver for a complete graph, that is one directed edge per pair of vertices. A path is then equivalently given by a sequence of vertices, that is simply a {\em word in vertices}.
 In particular, the complete graph on two vertices $\{0,1\}$, yields the quiver $\lefttorightarrow 0 \leftrightarrows 1\righttoleftarrow$ whose path algebra underlies Goncharov's and Brown's Hopf algebras for polyzetas \cite{BrownICM,Gont}. This is also the fundamental path groupoid for $\mathbb{C}\setminus\{0,1\}$ with tangential base points and is directly linked, cf.\ \cite[Section~1]{HopfPart1}, to Chen iterated integrals~\cite{ChenIterated}. The case with many vertices corresponds to polylogs.
The bi-algebra structure is actually founded in a simplicial structure, see \cite[Section~4]{HopfPart1}, \cite[Section~3.3.1]{HopfPart2}.
\end{Remark}

\subsubsection{Incidence coalgebra}
Another big source of inspiration and examples comes from combinatorics via {\em incidence coal\-ge\-bras}, as studied in \cite{JR,Schmitt}.
This is a coalgebra on the free $R$ module on the set of intervals~$[x,y]$
\begin{gather*}
\D([x,y])
=[x,x]\ot [x,y] +\sum_{z\colon x<z<y}[x,z]\ot [z,y] + [x,y]\ot [y,y].
\end{gather*}
The co-unit evaluates to $1$ on $[x,x]$ and $0$ else.

 \subsubsection{Categorical coalgebra}
These two examples are special cases of coalgebras stemming from a category with finite decomposition,
which will be the main case of interest in the applications.
Let $\C$ be a small category. This means that both the objects $X=\Obj(\C)$ and the morphisms $M=\Mor(\C)$ are sets.
The mapping $x$ to $\id_x$ identifies $X$ with a subset $X\subset M$ and $M=X\amalg \overline{M}$.
The free $R$ modules split accordingly $\Free{M}=\Free{I}\oplus \Free{\overline M}$.

Furthermore, restrict to the case where $\C$ is decomposition finite, this is for each $\phi\in M$, there are only finitely many pairs
$(\phi_0,\phi_1)\in M\times M$ such that $\phi_0\cdot \phi_1=\phi_1\circ\phi_0=\phi$.

The {\em categorical monoid coalgebra} $C[M]$ is defined to be the free $R$-module $\Free{M}$ with coproduct for $\phi\colon x\to y$ given by
\begin{gather*}
\Delta(\phi)=\sum_{(\phi_0,\phi_1)\colon\phi_0\cdot \phi_1=\phi}\phi_0\ot \phi_1=
\id_x \ot \phi + \phi\ot \id_y +\sum_{(\phi_0,\phi_1)\in \overline{M}\times \overline{M}\colon\phi_0\cdot \phi_1=\phi }\phi_0\ot\phi_1
\end{gather*}
and counit $\eps(\id_x)=1$ for the identity maps and $\eps(\phi)=0$ if $\phi$ is not an identity map.
So, let\-ting~$\eps_X$ be the projection onto $\Free{X}$, $\eps$ factors through $\eps_X$.

These coalgebras were analyzed in \cite{HopfPart1,HopfPart2,feynman}, and their history goes back to \cite{Moebiusguy}, see also~\cite{JR}.

\begin{Remark}
Note that there are two equivalent, opposite, ways to write down the composition maps,
which correspond to the two ways to write the compositions
$\Mor \stimest \Mor \to \Mor$ or $\Mor \bisub{t}{\times}{s} \Mor\to \Mor$
\begin{gather}
\Hom(X,Y)\times \Hom(Y,Z)\to \Hom(X,Z), \qquad
(\phi_0,\phi_1)\mapsto \phi_0\cdot \phi_1:= \phi_1\circ \phi_0,
\label{eq:monoidal}
\\
\Hom(Z,Y)\times \Hom(X,Y)\to \Hom(X,Z), \qquad
(\phi_0,\phi_1)\mapsto\phi_0\circ \phi_1.
\label{eq:catogorical}
 \end{gather}
We will call the first monoidal and the second categorical.
The two coproducts are opposites, i.e., $\Delta$ and $\Delta^{\rm op}=\tau\circ \Delta$, where $\tau$ is the flip.
The first version of this for the coproduct is what is used for posets and quivers and fits with the Connes--Kreimer coproduct \cite{CK,HopfPart1,HopfPart2},
where the subgraph is on the left and the cograph is on the right, cf. \eqref{eq:graphcoprod}. The second one is
more natural from a category point of view and corresponds to the tree Hopf algebra, where the stump is on the left and the branches
on the right, cf.\ \eqref{eq:CKcoprod}.

Note that this ambiguity is non-essential for a coalgebra or a bialgebra.
For a Hopf algebra, it might {\it a priori} make a difference,~-- $H^{\rm op,cop}$ is a Hopf algebra, but $H^{\rm cop}$ may not be~--
but if the Hopf algebra comes from a category $H^{\rm cop}$, is also Hopf, cf.\ Proposition \ref{prop:opprop}.
\end{Remark}

\begin{Remark}
Path and incidence coalgebras arise from categories as follows: For $P(Q)$ the objects are the vertices of $Q$ and $M=P(Q)$.
The source $s$ and target map $t$ map a path to its start and end vertices, respectively. The composition is the concatenation of paths, and the identities are the length $0$ paths.
 This is the free category generated by the morphisms corresponding to the edges.

 A poset $(S,\prec)$ defines a category whose objects are $S$ and whose morphisms are defined as follows.
There is one exactly one morphism $\phi_{x\to y}$
between $x$ and $y$ if and only if $x\preceq y$. The category is locally finite when the poset is.
The categorical coproduct is
$\D(\phi_{x\to y})=\sum_{z\in [x,y]} \phi_{x\to y}\ot \phi_{y\to z}
$ or in the usual poset notation, where one identifies $\phi_{x\to y}$ with the interval $[x,y]$. The identities are $\id_x=[x,x]$.

This category is the quotient of a quiver category.
The quiver has vertices $S$ and directed edges given by the
$(x,y)$-skew primitives. These are the elements $x\prec y$, $x\neq y$, where there is no $z\colon x\prec z \prec y$.
The quotient is given by identifying two morphisms, i.e., paths, $p$ and $q$ whenever they have the same source and target.
Categorically speaking this trivializes $\pi_1$ of the category making each component simply connected.
\end{Remark}

Algebraically speaking, a small category is the same as a colored monoid which is a particular type of partial monoid.

\begin{Definition}
A {\em colored monoid} $M$ is a set together with the following data. A set $X$ of colors, two morphisms $s,t\colon M\to X$, and a partial product
$\circ\colon M\stimest M\to M$
which is associative in the sense that $(ab)c$ exists, then so does $a(bc)$ and they coincide.

It is {\em unital} if there is morphism $\id\colon X\to M$, which is a section of both $s$ and $t$
such that $a\circ \id(t(a))=\id(s(a))\circ a=a$. $M=\amalg_{x,y}M_{x,y}$, where
 $M_{x,y}=\{a\,|\,s(a)=x,\,t(a)=y\}$.

A colored monoid is called {\em decomposition or locally finite}, if all the fibers $\circ^{-1}(m)$ for $m\in M$ are finite.

It is graded if there is a degree function $\deg\colon M\to \N_0$ such that $\operatorname{deg}(ab)=\operatorname{deg}(a)+\operatorname{deg}(b)$.
A degree function is proper is $\operatorname{deg}(a)=0$ if and only if $a$ is invertible.
\end{Definition}

The equivalence between colored monoids and categories is given by the identifications $X=\Obj{\C}$ and $M=\Mor(\C)$ with $M_{x,y}=\Hom(x,y)$.

We will need the following result \cite[Lemma 1.11]{HopfPart2}, which rephrased for colored monoids states that:
\begin{Lemma}
 In a decomposition finite colored unital monoid any left or right invertible element is invertible.
 \end{Lemma}

A categorical coalgebra is a dual construction to that of a colored monoid. In~the case of finite $X$ this is straightforward.
If $X$ is not finite there are subtleties which are relegated to Appendix \ref{sec:duals}.
The class of the resulting coalgebras is codified as colored coalgebras.

\subsubsection{Complications from isomorphisms}
\label{isopar}
 A colored unital monoid is a groupoid, if all of its elements are invertible in the colored sense, viz\ there is a morphism ${}^{-1}\colon M\to M$, such that
$\phi^{-1}\circ \phi=\id_{s(\phi)}$ and $\phi\circ \phi^{-1}=\id_{t(\phi)}$.
As~a~category, this means that all morphisms are isomorphisms.
Given a category $\C$, the underlying groupoid $\operatorname{Iso}(\C)$ is defined by the objects of $\C$ with only the isomorphisms.
This groupoid acts from the left and right on morphisms by conjugation $\phi\to \sds(\phi):=\sigma'
\circ\phi\circ \sigma^{-1}$.

\begin{Lemma}
\label{lem:discrete}
The product of $M$ being locally finite necessitates that there are only finitely many objects in each isomorphism class and that all automorphism groups are finite.

An identity morphism $\id_x$ is grouplike in $C[M]$, if and only if $x$ is the only element in its isomorphism class and it has no non-trivial automorphisms.

If there is a skew-primitive morphism $\phi\colon x\to y$, then $x$ and $y$ are the only elements in their isomorphism class and both have no non-trivial automorphisms.
\end{Lemma}
\begin{proof}
For an identity $\id_x$ the deconcatenation coproduct
will have a term $\sigma\ot \sigma^{-1}$ of any $\sigma\in \operatorname{Iso}(x,-)$, which proves the first two statements.

For a decomposition finite noninvertible $\phi\colon x\to y$ in the monoidal convention \eqref{eq:monoidal} the coproduct is
 \begin{equation}
\label{isocoprodeq}
\D(\phi)=\sum_{\s_x\in \operatorname{Iso}(X,-), \sigma_y\in \Iso (Y,-) }\big[\s_x\ot \s_x^{-1} \cdot \phi + \phi \cdot \s_y\ot \s_y^{-1}\big]+\cdots,
\end{equation}
where the first summands are always present.
The element $\phi$ being skew-primitive means that all the terms except $\phi\ot \id_y+\id_x\ot \phi$ are not present, which is the third statement.
\end{proof}

 We will call a morphism $\phi$ {\em essentially indecomposable} if the only decompositions into two factors have at least
 one factor which is an isomorphism. For an identity these are the terms $\sigma\ot \sigma^{-1}$,
 and if $\phi$ is not an identity, the corresponding terms of~$\Delta(\phi)$ are the displayed terms in~\eqref{isocoprodeq}.

Taking isomorphism classes will make it possibly to remedy the situation arising from too few grouplikes and skew-primitives
in case the action of $\operatorname{Iso}(\C)$ behaves nicely; which it does for a Feynman category,
see Section~\ref{channelpar}.

If there is only one object in each isomorphism class of objects, i.e., $\C$ is skeletal, then the sum will only be over automorphisms. In~the case of a finite groupoid with just one color, i.e., $M=G$ is a finite group, the deconcatenation coproduct is simply the familiar
\begin{gather*}
\Delta(g)=\sum_{h\in G} gh\ot h^{-1}=\sum_{h\in G} h\ot h^{-1}g.
\end{gather*}

\subsection{Colored co- and bialgebras}\label{sec:colored}
\subsubsection{Colored coalgebras}
The coalgebra on $X$, denoted by $C[X]$, is the coalgebra structure on $\Free{X}$, defined by letting all generators be grouplike, i.e., $\Delta(x)=x\ot x$ and $\eps(x)=1$.
Such a coalgebra, viz.\ a coalgebra freely generated by grouplikes, is often called {\em setlike}.

\begin{Lemma}\label{lem:grading}
A right $($or left$)$ coaction $\rho$ by $C[X]$ on an $R$ module $M$ is equivalent to grading by $X$, that is $M=\bigoplus_{x\in X}M_{x}$. Similarly having a right and a left coaction is the same as a~double grading by $X$, that
is $M=\bigoplus_{(x,y) \in X\times X}M_{x,y}$.
\end{Lemma}

\begin{proof}Given a coaction, set $M_{x}=\{m\,|\, \rho(m)=m\ot x\}$. Each
 $\rho(m)=\sum_x m_x\ot x$ is a finite sum and thus the module is the sum of the $M_x$. Furthermore $\D(x)=x\ot x$ so that $(m_x)_y=\delta_{x,y}m_x$.
and the sum is direct. Vice-versa setting $\rho(m_x)=m\ot x$ for $m_x\in M_x$ defines a coaction. One can proceed similarly for a left coaction.
The fact that left and right coactions commute proves the last statement.
\end{proof}

 For a coalgebra, given a left comodule $N$, $\lambda\colon N\to N\ot C$ and a right comodule $M$, $\rho\colon M\to C\ot M$
the cotensor product $M \bisub{\lambda}{\square}{\rho} N$ is defined as the coequalizer
\begin{equation*}
\begin{tikzcd}[column sep=large]
	M \bisub{\lambda}{\square}{\rho} N \arrow[r] & M \otimes_R N \arrow[r, "\rho_R \otimes_R \id_N", shift left] \arrow[r, "id_M \otimes_R \rho_L"', shift right] & M \otimes_R C \otimes_R N.
\end{tikzcd}
\end{equation*}
This is generated by elements $c\times c'$ with $\rho(c)=\lambda(c')$ which is dual to the condition that $t(m_1)=s(m_2)$.

\begin{Definition}
A {\em colored coalgebra $C$ with colors $X$} is a coalgebra $C$ together with a bi-co\-module structure of $C$ over $C[X]$, that is
$\lambda\colon C\to C[X]\ot C$, $\rho\colon C\to C[X] \ot C$ such that
 the comultiplication map is a morphism in the category of $C[X]$ bi-comodules and that $\D$ factors through the cotensor product:
\begin{equation*}
\xymatrix{
C\ar[rr]^\D\ar[dr]&&C\ot C\\
&C \bisub{\lambda}{\square}{\rho} C.\ar@{^(->}[ur]&
}
\end{equation*}
When omitting the mention of the colors, we implicitly assume the coalgebra is colored by its semigrouplike elements.
A colored coalgebra with colors $X$ is {\em color counital}, if there is a map of~$C[X]$ bi-comodule co-algebras $\eps_X\colon C\to C[X]$ such that
\begin{gather*}
(\eps_X\ot \id) \D = \lambda \qquad \text{and} \qquad (\id\ot \eps_x)\D=\rho,
\end{gather*}
and color coaugmented if there is a coalgebra map of $C[X]$ bi-comodules $i_X\colon C[X]\to C$ splitting~$\eps_X$.
 A graded color coaugmented coalgebra colored by $X$ is {\em color reduced} if $C_0=C[X]$.
\end{Definition}

By Lemma~\ref{lem:grading} a coalgebra colored by $X$ decomposes as $C=\bigoplus_{(x,y)\in X\times X} C_{x,y}$.
The condition of begin a map of $C[X]$ bi-comodules factoring through the cotensor product means that
\begin{equation*}
\D(C_{x,y})\subset \bigoplus_{z} C_{x,z}\boxtimes C_{z,y}.
\end{equation*}
If $C$ is color coaugmented,
$C=C[X]\oplus \overline{C}$,
where $\overline {C}=\ker(\eps_X)$ with $i_X(x)\in C_{x,x}$.
 Moreover,
$C_{x,x}= Ri_X(x) \oplus \overline {C}_{x,x}$. Setting $\overline C_{x,y}=C_{x,y}$ this means that $\overline{C}=\bigoplus_{(x,y)\in X\times X}\overline C_{x,y}$.
Since $\eps_X$ is a coalgebra map $\eps|_{\overline{ C}}=0$.

\begin{Remark}

If $C$ is co-augmented as a coalgebra, setting $X=\{e\}$, $\lambda (c)=e\ot c$ and $\rho(c)=c\ot e$ defines a coloring with one color $C_{e,e}=C$.
Thus, a color coaugmented coalgebra colored by~$X$ generalizes the notion of a coaugmented coalgebra to many colors corresponding to sets of grouplike elements.
\end{Remark}

\begin{Definition}
For a color coaugmented coalgebra colored by $X$, we define the {\em reduced dia\-go\-nal} to be
$
\bD =\Delta-\tilde\lambda- \tilde\rho
$,
where $\tilde \lambda = (i_X\ot \id)\lambda= i_X\eps_X\ot \Delta$ and $\tilde \rho = (\id\ot i_X)\rho= (\id\ot i_X\eps_X)\D$.
\end{Definition}
Note that $\pi_X=i_X\eps_X$ is the projection to the factor $C[X]$.

\begin{Definition} 	The \textit{Quillen filtration} of a color coaugmented coalgebra $C$ colored by $X$ is defined by:
 \begin{equation} \label{eq:QF}
 \FQ_0C= C[X], \qquad
\FQ_rC = \big\{x \in C \,|\, \bD\in \FQ_{r-1}C \boxtimes \FQ_{r-1}C\big\}.
\end{equation}
If this filtration is exhaustive $C = \bigcup_{i \geq 0}F^Q_iC$, then the coalgebra will be called {\em color connected}.
\end{Definition}
If $k$ is a field and $C$ is colored by $e$, this agrees with original definition of \cite{Quillen}, thus in the case of just one color,
we will say that $C$ is {\em Quillen connected} precisely if it is color connected.
A~graded coalgebra is Quillen connected if the degree zero part is isomorphic to $R$; in the literature~$R$ is often assumed to be a field $k$.

\begin{Theorem}\label{thm:colorstructure}
 Let $C$ be a colored coaugmented coalgebra colored by $X$. Then, for every $c \in \overline{C}_{x,y}$, $\Delta(c)$ can be written as $x \otimes c + c \otimes y + w$ with $w \in \overline{C} \boxtimes \overline{C}$, that is $w= \bD(X)$.

Additionally assuming that $C=\bigoplus_{i \geq 0}C_i$ is graded and color reduced, and thus also $\overline{C}$ is graded, we moreover have that
for $c \in \overline{C}_n$ and $n > 0$, $w \in \bigoplus_{i = 1}^{n - 1} C_i \boxtimes C_{n-i}$.
\end{Theorem}
\begin{proof}
First, say that $c\in \overline{C}_{x,y}$ with $x\neq y$.
We decompose (omitting the additional Sweedler sum symbols)
\begin{gather*}
 \Delta(c)=\sum_t c_{x,t}\Sp \ot c_{t,y}\Spp=\big(r\Sp x + \bar c_{x,x}\Sp\big)\ot \bar c_{x,y}\Spp
+ \bar c_{x,y}\Sp \ot
 \big(r \Spp y+\bar c_{y,y}\Spp\big)
+\sum_{z \notin \{x,y\}} \bar c_{x,z}\Sp\ot \bar c_{z,y}\Spp,
\end{gather*}
where we used the decomposition $C_{x,x}=Rx\oplus \overline {C}_{xx}$ and $C_{x,y}=\overline{C}_{x,y}$ for $x\neq y$.
Applying the left counit constraint $(\eps\ot \id)\circ \Delta=\id$ to $c$, we see that $r\Sp \bar c_{x,y}\Spp=c $ and the right unit constraint
gives $\bar c_{x,y}\Sp r\Spp=c$.
For $c\in \overline C_{x,x}$
\begin{gather*}
 \Delta(c)=\sum_t c_{x,t}\Sp \ot c_{t,x}\Spp=\big(r\Sp x + \bar c_{x,x}\Sp\big)\ot \big(r\Spp x + \bar c_{x,x}\Spp\big)
 +\sum_{z \neq x} \bar c_{x,z}\Sp\ot \bar c_{z,y}\Spp.
\end{gather*}
From the left unit constraint, we obtain that $c=r \Sp r\Spp x + r \Sp \bar c_{x,x}\Spp$, but as $\eps(c)=0$, $r\Sp r\Spp=0$ and $r\Sp\bar c_{x,x}\Spp=c$.
Similarly from the right unit constraint $\bar c_{x,x}\Sp r\Spp=c$.

In the graded case $\D(C_n)\subset \bigoplus_{i=0}^n C_i \bot C_{n-i}$ and
 $w \in \bigoplus_{i =1}^n\overline{C}_i \bot \overline{C}_{n-i} \subset \bigoplus_{i=0}^n C_i \bot C_{n-i}$, but $\overline{C}_0=0$, so the first and the last summand vanish.
\end{proof}

\begin{Lemma} In a color coaugmented coalgebra colored by $X$
the reduced comultiplication is coassociative, and thus there are unique $n$-th iterates $\bD^{[n]}\colon C\to C^{\ot n}$.
\end{Lemma}

\begin{proof}Because of linearity, it suffices to show it holds for $c \in C_{x,y}$.
Using the following abbreviated Sweedler notation for the comultiplication $\Delta(c) = c\Sp\otimes c\Spp$
and $\D^{(3)}(c)=c\Sp \otimes c\Spp \otimes c\Sppp$, calculating the left and right hand sides yields:
\begin{align*}
\big(\bar{\Delta} \otimes \id_C\big) \circ \bar{\Delta}(c)
={}&\big(\bar{\Delta} \otimes \id_C\big) \circ \bar{\Delta}(c)
= \big(\bar{\Delta} \otimes \id_C\big)\big(\Delta - \tilde\lambda - \tilde\rho\big)(c)
\\
= {}&\big(\big(\Delta - \tilde\lambda - \tilde\rho\big) \otimes \id_C\big) (\Delta(c) - x \otimes c - c \otimes y)
\\
={}&\big(\Delta - \tilde\lambda - \tilde\rho\big)\big(c\Sp\big) \otimes c\Spp - \big(\Delta(x) - \tilde\lambda(x) - \rho(x)\big) \otimes c
\\
&- \big(\Delta(c) - \tilde\lambda(c) - \tilde\rho(c)\big) \otimes y
\\
={}& \Delta\big(c\Sp\big) \otimes c\Spp - x \otimes c\Sp \otimes c\Spp - \tilde\rho\big(c\Sp\big) \otimes c\Spp
\\
&- (x \otimes x - x \otimes x - x \otimes x) \otimes c
- \big(c\Sp \otimes c\Spp - x \otimes c - c \otimes y\big) \otimes y
\\
={}&c\Sp \otimes c\Spp \otimes c\Sppp - x \otimes c\Sp \otimes c\Spp- \tilde\rho\big(c\Sp\big) \otimes c\Spp
\\
&+ x \otimes x \otimes c - c\Sp \otimes c\Spp \otimes y
+ x \otimes c \otimes y + c \otimes y \otimes y,
\\
\big(\id_C \otimes \bar{\Delta}\big) \circ \bar{\Delta}(c)={}&\big(\id_C \otimes \bar{\Delta} \big) \circ \bar{\Delta}(c)
	= (\id_C \otimes \bar{\Delta})\big(\Delta - \tilde\lambda - \tilde\rho\big)(c)
\\
={}& \big(\id_C \otimes \big(\Delta - \tilde\lambda - \tilde\rho\big) \big) (\Delta(c) - x \otimes c - c \otimes y)
\\
={}&c\Sp \otimes \big(\Delta - \tilde\lambda - \tilde\rho\big)\big(c\Spp\big) - x \otimes \big(\Delta(c) - \tilde\lambda(c) - \tilde\rho(c)\big)
\\
&- c \otimes \big(\Delta(y) - \tilde\lambda(y) - \tilde\rho(y)\big)
\\
={}&\big(c\Sp \otimes c\Spp \otimes c\Sppp - c\Sp \otimes \tilde\lambda\big(c\Spp\big) - c\Sp \otimes c\Spp \otimes y\big)
\\
&- x \otimes \big(c\Sp \otimes c\Spp - x \otimes c - c \otimes y\big)
- c \otimes (y \otimes y - y \otimes y - y \otimes y)
\\
={}&c\Sp \otimes c\Spp \otimes c\Sppp - c\Sp \otimes \tilde\lambda\big(c\Spp\big) - c\Sp \otimes c\Spp \otimes y - x \otimes c\Sp \otimes c\Spp
\\
&+ x \otimes x \otimes c + x \otimes c \otimes y + c \otimes y \otimes y.
	\end{align*}
	
These agree if $ \tilde\rho\big(c\Sp\big) \otimes c\Spp = c\Sp \otimes \tilde\lambda\big(c\Spp\big)$, which
readily follows the assumption of being color coaugmented.
\end{proof}

The lemma allows us to introduce the colored conilpotent coradical filtration by
\begin{equation}
\label{eq:conilfilt}
\FN_0C=i(C[X]), \qquad
\FN_rC =i(C[X])\oplus \big\{x \in C \,|\, \bD^{[r+1]}(x) = 0\big\},\qquad
r\geq 1.
\end{equation}
This is the generalization of the conilpotent coradical filtration, see, e.g., \cite{Loday}. \begin{Definition}
We call a color coaugmented co-algebra colored by $C[X]$ {\em conilpotent} if the filtration $\FN C$ is exhaustive.
A color coaugmented coalgebra colored by $X$ is called {\em $\bD$ flat}, if~for all $n$, $\ker\big(\bD^{[n]}\ot \id\big)= \ker\big(\bD^{[n]}\big)\ot C$ and
and $\ker\big(\id\ot \bD^{[n]}\big)=C\ot \ker \big(\bD^{[n]} \big)$.
\end{Definition}

 For a coaugmented coalgebra with one color $e$,
 the filtration coincides with the {\em conilpotent coradical filtration}
and color conilpotence agrees with notion of conilpotence for coaugmented coalgebras, see, e.g., \cite{Loday}.

\begin{Corollary}
\label{cor:gradedcolornil}
A graded color coaugmented color reduced coalgebra colored by $X$ is conilpotent.
\end{Corollary}
\begin{proof}
This follows directly from Theorem \ref{thm:colorstructure}.
\end{proof}

The notions of color conilpotent and color connected are related.
Color conilpotence is easier to check, but color connectedness is better to argue with.

\begin{Proposition}\label{prop:nilconnected}
A color connected coalgebra is a color conilpotent coalgebra and a $\bD$-flat conilpotent coalgebra
 is color connected.
 \end{Proposition}

 \begin{proof} If $C$ is color connected, say $c\in F^Q_n(C)\cap \overline{C}$, by induction there is some $N$ such that $\bD^{[N]}(c)=0$.
 This is clear for $c\in F_0^QC$ and if $c\in F^Q_{n+1} C\cap \overline{C}$ then $\bD(c)\in F^Q_nC\ot F^Q_n C$ so there is some $N$ such
 that $\big(\bD^{[N]}\ot \id\big)\bD(c)=\bD^{[N+1]}(c)=0$.
Vice-versa, assume that $C$ is $\bD$-flat color conilpotent. Again using induction, as $\FN_0 C= F_0^QC$, we can assume that $\FN_{n-1}C\subset F^Q_{n-1}C$. Now, if $c\in \ker\big(\bDn\big)$, then, since $C$ is $\bD$-flat, $\bD(c)\in \ker\big(\bD^{[n-1]}\ot \id\big)\cap \ker\big(\id\ot \bD^{[n-1]}\big)=
 \ker\big(\bD^{[n-1]}\big)\ot C\cap C\ot \ker\big(\bD^{[n-1]}\big)\subset F^Q_{n-1}F\ot F_{n-1}^QC$ and $c\in F^Q_n C$.
 \end{proof}

\begin{Lemma}\label{lem:bdflat}
 A color coaugmented coalgebra $C$ colored by $X$
is $\bD$-flat if $C$ is flat. In~particular, $\bD$ flatness is automatic over a field or if $C$ is free.
\end{Lemma}
\begin{proof}
Consider the short exact sequence $0\to \ker\big(\bD^{[n]}\big)\to C\to {\im}\big(\bD^{[n]}\big) \to 0$. By right exactness of $\ot$:
$ \ker\big(\bD^{[n]}\big)\ot C \to C\ot C \to {\im}\big(\bD^{[n]}\big) \ot C\to 0$ is exact. If $C$ is flat, $0\to \im\big(\bD^{[n]}\big) \ot C^{\ot n}\to C\ot C$ is exact and hence
$ \ker\big(\bD^{[n]}\big)\ot C \stackrel{i\ot \id}{\to} C\ot C \stackrel{\bDn \ot \id}\to \C^{\ot n} \ot C$ is exact and $\ker\big(\bDn \ot \id\big)=\ker\big(\bDn\big)\ot C$.
\end{proof}

\begin{Proposition}\label{prop:catcolor}
For a decomposition finite colored monoid $M,$ the categorical coalgebra $C[M]$ is colored by its objects and color coaugmented.

If $M$ only has identities as invertibles, then $C[M]$ is colored by its grouplikes.
If additionally, $C[M]$ is color nilpotent, then it is color connected.

In particular, if $M$ has a proper degree function, is locally finite and has no non-trivial identities, then it is colored by its grouplikes and is color connected.
\end{Proposition}

\begin{proof}
The first statement is clear by construction. If identities are the only invertible elements,
then by Lemma \ref{lem:discrete} they are grouplike, and
these are also the only semigrouplike elements. Hence, $C[M]$ is colored by its grouplikes.
As $C[M]$ is free being color nilpotent implies being color connected by Lemma \ref{lem:bdflat}
 and Proposition \ref{prop:nilconnected}.
The last statement follows from the previous ones and Corollary \ref{cor:gradedcolornil}.
 \end{proof}

\subsubsection{Colored bialgebras and the free bialgebra \textit{\textbf{FB}}[\textit{\textbf{M}}] on a colored monoid}

Generally, we will say that a bialgebra has an attribute like colored, color connected etc.\ if its coalgebra does.

\begin{Lemma}Given a colored connected bialgebra colored by $X$ the ideal $I$ generated by $x-y$, for $x,y\in X$ is a coideal and $C/I$ is Quillen connected.
\end{Lemma}
\begin{proof}
Indeed, $\eps(x-y)=1-1=0$ and $\Delta(x-y)=x\ot x - y\ot y = x\ot (x-y)+(x-y)\ot y$, so $I$ is a coideal. $C/I$ is single colored and coaugmented as $C[X]/I=R$.
The induced filtration $F^Q_r( C/I)=F^Q_rC/\big(I\cap F^Q_rC\big)$ is still exhaustive.
\end{proof}

The following construction is key for many examples. In~particular for constructing bialgebras from posets or quivers.

For a set $X$ let $X^\times=\amalg_{n\in \N_0}X^{\times n}$
be the free monoid on $X$. The unit $1$ is represented by the empty product.
If $M$ is a monoid then $M^\times$ additionally carries a componentwise mono\-idal structure:
$(a_1\kdk a_n)(b_1\kdk b_n)=(a_1b_1\kdk a_nb_n)$.
Consider $\Free{M^\times}$ as the
 {\em free unital algebra} on $M^\times$. If $M$ is decomposition finite,
this has a coalgebra structure induced by that of $C[M]$, whose comultiplication is given by $\Delta(1)=1\ot 1$ and $\D((a_1\kdk a_n))=\sum \big(a_1\Sp\kdk a_n\Sp\big) \ot \big(a_1\Spp\kdk a_n\Spp\big)$, where $\D(a_i)=\sum a_i\Sp\ot a_i\Spp$ is the coproduct of~$C[M]$. The counit is
 $\eps((a_1\kdk a_n))=\eps(a_1)\cdots \eps(a_n)$.
These structures make $\Free{M^\times}$ into a unital counital bialgebra with counit $\eps$ and unit $1$, which we will denote by $FB[M]$.

\begin{Proposition}
If $M$ is a monoid colored by $X$, then $M^\times$ is colored by $X^\times$. In~this case the underlying algebra structure of $FB[M]$ is graded by the monoid $X^\times$ and
the underlying coalgebra of $FB[M]$ is a $X^\times$ colored, and Proposition~$\ref{prop:catcolor}$ applies.
\end{Proposition}

\begin{proof}
The coloring is given by $s(a_1\kdk a_n)=(s(a_1)\kdk s(a_n))$ and $t(a_1\kdk a_n)=
(t(a_1),\allowbreak\ldots, t(a_n))$. The computations
are then straightforward.
\end{proof}

In this construction the original composition is turned into a coproduct. Then a free product is added to make a bialgebra.
If $M$ is interpreted as a category $\C$, this is the categorical coalgebra for the free monoidal category $\C^\ot$.
The case were there already is a monoidal structure, which is not necessarily free,
is treated in Section~\ref{catpar}.

\section[Path-like co- and bialgebras, convolution inverses and~antipodes]
{Path-like co- and bialgebras, convolution inverses \\and~antipodes}

\subsection{QT-filtration and convolution inverses}
\label{takeuchipar}

In order to define antipodes, or more generally convolution inverses,
it is useful to regard filtrations, as in good cases one can recursively build up the antipode from its value on the degree~$0$ part.
To this end, we put forward the definition of a QT-filtration (Quillen--Takeuchi)
that generalizes the classical graded, coalgebra, coradical, Quillen \cite{Quillen} filtrations and their colored conilpotent versions.
This notion is what is needed to prove a general version of Takeuchi's Lemma \cite{Takeuchi}, for the existence of convolution inverses.
\begin{Definition}
	A {\em QT-filtration} of a coalgebra $C$ is a filtration of $C$ by $R$-submodules $\FQT_p(C)$, $p\geq 0$ such that:
\begin{enumerate}\itemsep=0pt
\item $\FQT_0C=S$ is a subcoalgebra.
\item $\Delta\big(\FQT_pC \big)\subset \FQT_{p-1}C \boxtimes C + C\boxtimes \FQT_{p-1}C $.	
\end{enumerate}
We call a QT-filtration a {\em QT-sequence} if the filtration is exhaustive, i.e., $C = \bigcup_{i=0}^\infty F_iC$ and call the QT-sequence {\em split},
if $\FQT_0C$ is a direct summand, that is $0\to \FQT_0C\to C\to \FQT_0C/C\to 0$ is split.

A coaugmented coalgebra is \em{QT-connected} if is has a QT-sequence whose degree $0$ part is ${\rm Re}$.
\end{Definition}

\begin{Remark}
Note that a QT-filtration is defined by specifying a base, that is a subcoalgebra~$S$ and setting:
 $\FQT_0C=S$ and $\FQT_rC = \big\{x\,|\,\Delta(x) \in C \boxtimes \FQT_0C + \FQT_{p-1}C \boxtimes \FQT_{p-1}C + \FQT_0C \boxtimes C\big\}$.	
Quillen's original filtration~\cite{Quillen} for a coaugmented coalgebra is obtained from $S=\FQT_0={\rm Re}$. In~the category of $R$-coalgebras, every coalgebra filtration gives rise to a QT-filtration with $S$ being the degree $0$ part.

The filtration associated with the grading of an $\mathbb{N}$-graded coalgebra \big($\FQT_r := \{x\,|\, \deg x \leq r\}$\big) is a QT-filtration based on $\FQT_0$, because $\FQT_i \subset \FQT_r$ if $i \leq r$. $\FQT_i \boxtimes \FQT_j \subset \FQT_r \boxtimes \FQT_q$ if $i \leq r$ and $j \leq q$.
Given a $\Z$ graded filtration, with $C_0$ a subcoalgebra, and with $\Delta(F_p)\subset F_{p-1}\boxtimes C + C\boxtimes F_{p-1}$, $p>0$ and
$\Delta(F_p)\subset F_{p+1}\boxtimes C + C\boxtimes F_{p+1}$, $p< 0$,
setting $\FQT_p=F_p+F_{-p}$, $p \geq 0$ produces a QT-filtration.
\end{Remark}

\begin{Theorem}\label{mainthm}
 Consider a coalgebra $C$ with counit $\eps$ and an algebra $A$ with unit $\eta$.
Given a QT-sequence $\big\{\FQT_i\big\}^{\infty}_{i=0}$ on $C$, such that ${\rm Ext}^1_R\big(C/\FQT_0C,A\big)=0$,
for any element $f$ of $\Hom(C, A)$, $f$~is $\star$-invertible if and only if
the restriction $f|_{\FQT_{0}C}$ is $\star$-invertible in $\Hom\big(\FQT_{0}C, A\big)$.
\end{Theorem}
\begin{proof}
First assume that $f$ is convolution invertible in $\Hom(C,A)$, and let $g$ be $\star$-inverse of $f$. By restriction to $\FQT_0C$,
see Lemma \ref{restrictionlem},
it follows that $(\eta \circ \epsilon)|_{\FQT_{0}C} = (f \star g)|_{\FQT_{0}C} = f|_{\FQT_{0}C} \star g|_{\FQT_{0}C}$ and
$(\eta \circ \epsilon)|_{\FQT_0C}$ is the unit of the restricted convolution algebra $\Hom\big(\FQT_{0}C, A\big)$. Thus $g|_{\FQT_0C}$ is a right $\star$-inverse
of $f|_{\FQT_0}$. By symmetry, it is also a left $\star$-inverse and thus the $\star$-inverse in the restricted algebra.
	
Conversely, assume that $f|_{\FQT_0C}$ is convolution invertible in $\Hom\big(\FQT_{0}C, A\big)$. Thus, there is an element
$g$ of $\Hom\big(\FQT_{0}C, A\big)$ such that $f\star g=\eta\circ\epsilon=g\star f$ on $\FQT_{0}C$.
We can extend the $R$-linear map $g$ to $C$ as the restriction morphisms $\Hom(C,A)\to \Hom\big(\FQT_0C,A\big)$
is surjective due to the assumption that $\mathop{\rm Ext}\big(C/\FQT_0C,A\big)=0$.
Because the convolution algebra $(\Hom(C, A), \star, \eta \circ \epsilon)$ is a~unital monoid, if we show that
both $f \star g$ and $g \star f$ have a~$\star$-inverse this will imply $f$ has a~$\star$-inverse.
Thus, replacing $f$ with $f\star g$, we may assume $f|_{\FQT_0C} = \eta \circ \epsilon|_{\FQT_0C}$
that is $\eta \circ \epsilon |_{\FQT_0C} - f|_{\FQT_0C}= 0$.
Since the filtration is exhaustive, for every $x \in C$, there exists a natural number $N$
such that $x \in F_N$ and $(f - \eta \circ \epsilon)^{\star (N)}(x) = 0$.
Indeed, after using the coassociativity for the diagonal to decrease the filtration degree,
every summand from $\Delta^{[N+1]}(x) = \sum_{x} x_{(1)}x_{(2)}x_{(3)}\cdots x_{(N+1)}$ has to contain a factor in $F_0C$
according to the definition of the QT-filtration, and hence $\eta \circ \epsilon - f$ annihilates each term.
So, the infinite sum $h=\sum^{\infty}_{i=0} (\eta \circ \epsilon - f)^{\star i} $ is locally finite and
yields a~well defined map, where we set $(\eta \circ \epsilon - f)^{\star 0} = \eta \circ \epsilon$.
It is easy to verify that $h$ is $R$-linear, so $h \in \Hom(C, A)$. That $h$ is the inverse of $f$ follows by computation:
$f \star h = (f - \eta \circ \epsilon) \star h + \eta \circ \epsilon \star h = (f - \eta \circ \epsilon) \star h +
h = - \sum^{\infty}_{i=1} (\eta \circ \epsilon - f)^{\star i} + \sum^{\infty}_{i=0} (\eta \circ \epsilon - f)^{\star i}
= (\eta \circ \epsilon - f)^{\star 0} = \eta \circ \epsilon$. The fact that $h \star f = \eta \circ \epsilon$
can be checked in the same way. Therefore, we conclude $h$ is $\star$-inverse of $f$.
\end{proof}

Note that if $\FQT_0C$ is a direct summand, i.e., the sequence $0\to \FQT_0C\to C\to C/\FQT_0C\to 0$ splits, the Ext group vanishes, and the first condition is automatically met.

We say that a bialgebra has a QT-sequence, if the coalgebra has a QT-sequence. In~this situation, Theorem \ref{mainthm} yields the generalization of the results of Quillen and Takeuchi, and we recover many of the existence theorems for antipodes as special cases.

\begin{Theorem}\label{qtsplitthm}
A bialgebra with a QT-sequence and ${\rm Ext}^1_R\big(C/F_0^QC,C\big)=0$ is Hopf, i.e., has an antipode, if and only if $\id|_{\FQT_0C}$ has a $\star$ inverse.
\end{Theorem}

\begin{Corollary}\label{QTantipodecor}A bialgebra with a split QT-sequence is Hopf if and only if $\id_{F_0^QC}$ has a convolution inverse. In~particular, if $C$ is a split QT-connected bialgebra, then it is Hopf.
\end{Corollary}
\begin{proof}The first part directly follows from Theorem \ref{qtsplitthm}. If $C$ is QT-connected then $S(e)=e$ is the required antipode on $F_0^QC$.
\end{proof}

\begin{Proposition}\label{prop:coloredconil}
The Quillen filtration $\FQ C$, cf.\ \eqref{eq:QF}, for a color coaugmented coalgebra colored by $X$ is a QT-filtration and a QT-sequence if $C$ is connected.
For a $\bD$-flat color coaugmented coalgebra colored by $X$ the conilpotent filtration $\FN C$, cf.\ \eqref{eq:conilfilt},
is a QT-filtration with $\FQT_0C=C[X]$ which is exhaustive if $\bD$ is conilpotent.
\end{Proposition}
\begin{proof}
The first part follows from the fact that in this situation, $\bD=\bigoplus_{g,h \in X\times X}\D_{g,h}$ and thus $\FQT_rC\supset F^Q_r(C)$
with $\FQT_0C=C[X]=\bigoplus_{g\in \SGL(C)}Rg$.
 For the second part:
decompose $c\in F_r(C)$ as $c_0+\bar c\in C= C[X]\oplus \bar C$ with $\bar c$ in $\ker\big(\bD^{[r]}\big)$, then by Theorem~\ref{thm:colorstructure}
$\Delta(c)=c_0\ot c_0+ i_X\lambda(\bar c)\ot \bar c+ c \ot i_X\rho(\bar c)+\bD(\bar c)$ with $\bD(\bar c)\in \ker\big(\bD^{[r-1]}\big)\ot {\rm id}$. Thus, by definition of $\bD$-flatness
$\Delta(c)\in F_0C\boxtimes C\oplus C\boxtimes F_0C \oplus F_{r-1}\boxtimes C$.
\end{proof}

\begin{Corollary}
A bialgebra whose coalgebra is colored by $X$ and color connected, or
whose coalgebra is a $\bD$-flat color coaugmented coalgebra with colors $X$ has an antipode if and only if every generator $x\in C[X]$
has a algebra inverse $x^{-1}$.
\end{Corollary}
The Proposition allows to give a recursive definition in both cases.

\subsection{Sg-flavored bi- and coalgebras}
\label{sg-flavoredpar}
 In this section, we introduce a type of coalgebra which has a QT-filtration based on a decomposition according to $(g,h)$-primitives. This
 generalizes path-coalgebras and colored coalgebras and subsumes the notions of
graded connected coalgebras and May's component coalgebras \cite{MayMoreConcise}.
This type of idea is also used in \cite{TaftWilson}.

\begin{Definition}
	Let $C$ be a $R$-coalgebra and $g$ and $h$ be semigrouplike elements.
	We define the $(g,h)$-reduced comultiplication to be $\overline{\Delta}_{g,h}(x) = \Delta(x) - x \otimes g - h \otimes x$.
	\end{Definition}

\begin{Lemma}
\label{lem:associator}
Pairs of $\D_{g,h}$s have the following associators:
\begin{gather*}
\big(\overline{\Delta}_{g,h} \otimes \Id\big)\overline{\Delta}_{g,h}(x) -	\big(\Id \otimes \overline{\Delta}_{g,h}\big)\overline{\Delta}_{g,h}(x)\\
\qquad {} =x^{(1)} \otimes (h - g) \otimes x^{(2)} - x \otimes (h - g) \otimes g
- h \otimes (h -g) \otimes x.
\end{gather*}
The $\overline{\Delta}_{g}:=\D_{g,g}$ are coassociative,
thus, the iterated reduced comultiplication $\overline{\Delta}_{g}^{[n]}(x)$ is well-defined.
 Furthermore, if $g$ is grouplike and $x \in \ker\epsilon$, then $(\epsilon \otimes \Id)\overline{\Delta}_{g} (x)= (\Id \otimes \epsilon)\overline{\Delta}_{g} (x)= 0$.
\end{Lemma}

\begin{proof}
The first claim follows from the calculation:
\begin{gather*}
\big(\overline{\Delta}_{g,h} \otimes \Id\big)\overline{\Delta}_{g,h}(x) -	\big(\Id \otimes \overline{\Delta}_{g,h}\big)\overline{\Delta}_{g,h}(x)
\\ \qquad
{}=\sum_x 	x^{(11)} \otimes x^{(12)} \otimes x^{(2)} - \sum_x
x^{(11)} \otimes x^{(21)} \otimes x^{(22)} 	- x^{(1)} \otimes g \otimes x^{(2)}
\\ \qquad\hphantom{=}
{}- h \otimes x^{(1)} \otimes x^{(2)} - x^{(1)} \otimes x^{(2)} \otimes g + x \otimes g \otimes g + h \otimes x \otimes g + h \otimes g \otimes x
\\ \qquad\hphantom{=}
{} -\big[ x^{(1)} \otimes x^{(2)} \otimes g - x^{(1)} \otimes h \otimes x^{(2)} + x \otimes h \otimes g -h \otimes x^{(1)} \otimes x^{(2)} + h \otimes x \otimes g
\\ \qquad\hphantom{=}
{}+ h \otimes h \otimes x\big]
=x^{(1)} \otimes (h - g) \otimes x^{(2)} - x \otimes (h - g) \otimes g - h \otimes (h -g) \otimes x.
\end{gather*}
The second claim is straightforward, e.g., $(\epsilon \otimes \Id)\overline{\Delta}_{g}(x) = x -\epsilon(x)g - \epsilon(g)x = 0$.
\end{proof}

Given a coalgebra $C$, we will inductively define $R$-modules each depending on two semi\-group\-like elements $g,h\in \SGL(C)$. These are simultaneously defined for all pairs of such elements and yield pieces of a filtration.
\begin{enumerate}\itemsep=0pt
\item For any two $g,h\in \SGL(C)\colon F_0^QC_{g,h} := Rg \cap Rh$ is the base $R$-module.

\item Inductively over all pairs:
$F_n^QC_{g,h} := \big\{x\,|\, \overline{\Delta}_{g,h}(x)=
\in \sum_{t \in \SGL(C)} F_{n-1}^QC_{g,t} \boxtimes F_{n-1}^QC_{t,h}\big\}$.
\end{enumerate}
In words, $F_n^QC_{g,h}$ is the subset in which every element $x$ has the property that if $\overline{\Delta}_{g,h}(x)= \sum x^{(1)} \otimes x^{(2)}$ then for each summand $x^{(1)} \otimes x^{(2)}$ there exists a semigrouplike element $t$ with $x^{(1)} \in F_{n-1}^QC_{g,t}$ and $x^{(2)} \in F_{n-1}^QC_{t,h}$. This is also a $R$-module as every $F_n^QC_{g_1,g_2}$ with $g_1,g_2 \in \SGL(C)$ is a $R$-module inductively.
The \textit{$(g, h)$-Quillen component} of $C$ is defined to be $\FQ C_{g,h} = \sum_{i \geq 0} \FQ_iC_{g,h}$.

\begin{Lemma} The $\FQ_iC_{g,h}$ yield filtrations of the $\FQ C_{g,h}$; that is
	$F_iC^Q_{g,h} \subset F_{i+1}C^Q_{g,h}$ for $i \geq 0$.
\end{Lemma}
\begin{proof}
	We use induction. Base case: to show that $Rg \cap Rh = \FQ_0C_{g,h} \subset \FQ_{1}C_{g,h}$ pick $x\in \FQ_0C_{g,h}$, say $x=rg$, then $\overline{\Delta}_{g,h}(x)=-x\ot h\in F_0C_{g,h}^Q\boxtimes F_0C_{h,h}^Q$.	
	By induction assume that $F_{n-1}C^Q_{g,h} \subset F_{n}C^Q_{g,h}$. Let $x \in \FQ_{n}C_{g,h}$. Then $ \overline{\Delta}_{g,h}(x) \in \sum_{t \in \SGL(C)} \FQ_{n-1}C_{g,t} \boxtimes F_{n-1}C^Q_{t,h} \subset \sum_{t \in \SGL(C)} F_{n}C^Q_{g,t} \boxtimes F_{n}C^Q_{t,h}$ by induction hypothesis, and thus $x \in F_{n+1}C^Q_{g,h}$ .
\end{proof}

\begin{Definition}
The filtrations of $C^Q_{g,h}$ yield a common filtration $\FBQ_rC := \sum_{g,h} F^Q_rC_{g,h}$~-- the bivariate Quillen filtration.
We say a coalgebra $C$ is \textit{sg-flavored} (sg stands for semigrouplike) if the bivariate Quillen filtration is exhaustive, that is $C = \sum_{(g,h), g,h\in \SGL(C)} C_{g,h}^Q$ and {\em split} if~$\FBQ_0C$ is a direct summand.
A bi- or Hopf-algebra is sg-flavored if the underlying coalgebra is sg-flavored.
\end{Definition}

\begin{Remark}
\label{linermk}

Note all semigrouplike elements of $C$ lie in $\FBQ_0C$ and all skew primitive elements~$C$ lie in $\FBQ_1C$. The degree zero component
$F_0C^Q_{g,h}$ may be non-empty, if $g,h$ are two distinct semigrouplike elements, it may happen that there exist $s, t \in R$ such that $sg = th$ so that $Rg \cap Rh \not= 0$.
This complication does not appear over a ground field $R=k$ or if the respective submodules are free over $\mathbb{Z}$, as the equation then implies that $g = h$.

$C^Q_{g,h}$ is never empty. This is clear for $g=h$. For $g\neq h$ $\Delta(g - h) = g \otimes (g - h) + (g - h) \otimes h$ and hence
	$\overline{\Delta}_{g,h}(g-h)=0\in \FBQ_0C_{g,t}\boxtimes \FBQ_0C_{t,h}$ for any $t$.
	Thus, if there is more than one semigrouplike element, the sum $C = \sum_{g,h} C_{g,h}^Q$ is never direct, as $g-h$ is in $\big(C^Q_{g,g}+C^Q_{h,h}\big)\cap C^Q_{g,h}$.
	
	In the case of a field $R=k$, $g-h$ generates the intersection and in fact over a field one can
	 make the sum direct by reducing $C_{g,h}$ by $k(g-h)$ following the line of arguments presented in~\cite{Susan}. Note that the corresponding splittings are not unique.
\end{Remark}

\begin{Lemma}
\label{lem:sgQT}
The $R$-modules $\FBQ_rC := \sum_{g,h} F_rC_{g,h}^Q$ for $r \geq 0$, form a QT-filtration, which is a QT-sequence if $C$ is sg-flavored.
\end{Lemma}
\begin{proof}
By definition the elements in $\FBQ_0C$ are multiples of semigrouplike elements and hence $F_0C$ is a subcoalgebra.
Condition (2) for a QT-filtration is satisfied by construction. In~the sg-flavored case, the filtration is exhaustive by definition.
\end{proof}

{\sloppy
\begin{Theorem}\label{thm:sgflavor}
For an sg-flavored coalgebra with ${\rm Ext}^1_R\big(C/\FBQ_0C,A\big)=0$,
an element~$f$ of $\Hom(C, A),$ $f$ is $\star$-invertible if and only if
the restriction $f|_{\FQT_{0}C}$ is $\star$-invertible in $\Hom\big(\FBQ_{0}C, A\big)$.
	
In particular, a split sg-flavored bialgebra is a Hopf algebra if $\FBQ_0C$ is the quotient of a group algebra by a Hopf ideal.
\end{Theorem}

}

\begin{proof}
By Lemma \ref{lem:sgQT},
 Theorem \ref{qtsplitthm} applies and we need that $id$ has an antipode. Now on any semigrouplike
 there is only one possible value for the antipode:
 $S(g)=g^{-1}$, if it exists. Thus, $S$ being defined on $g$ is equivalent to $g$ being invertible.

If $g$ is grouplike, so is $g^{-1}$, so that $g^{-1}\in \FBQ_0C$. This follows from applying the counit constraints
to $1\ot 1=\D(1)=\D\big(gg^{-1}\big)=\sum g\big(g^{-1}\big)\Sp \ot g\big(g^{-1}\big)\Spp$.
 Thus for the second statement, the antipode is fixed on the generators and descends precisely if the quotient is by a~Hopf ideal.
\end{proof}

The following proposition explains the usual constructions of antipodes which pass through a connected quotient.

\begin{Proposition}
Let $B$ be a sg-flavored split bialgebra all of whose semigrouplikes are group\-like,
then the ideal $I$ spanned by
$g-h$ for $g,h\in\GLE(C)$ is a coideal and $B^{\rm red}=B/I$ is a Hopf algebra.
\end{Proposition}

\begin{proof}
As before $\eps(g-h)=1-1=0$ and $\D(g-h)=g\ot (g-h)+(g-h)\ot h$.
This filtration is exhaustive, as it is induced by the quotient $\FBQ_r(C/I)=\FBQ_r(C)/\big(I\cap \FBQ_r\big)$. The bivariate Quillen filtration on the quotient is simply the Quillen filtration. The splitting passes to the quotient and as a Quillen connected bialgebra $C/I$ is Hopf.
\end{proof}

It is possible to truncate to only the grouplike and skew primitive elements.
\begin{Proposition}
\label{prop:flavorgen}
	If a bialgebra $B$ is generated by grouplike and skew primitive elements, then it is sg-flavored.
\end{Proposition}
\begin{proof}
		Observe that if $g$, $h$ are grouplike elements and $x$ a skew primitive element, then $gh$ is grouplike and $gx$ is a skew primitive. Therefore,
	we focus on the products of skew primitives $x_n\cdots x_1$. If $x$ is a $(g, h)$-skew primitive and $x'$ is a $(g',h')$-skew primitive, it follows that
\begin{gather*}
\Delta(xx') - (gg' \otimes xx' + xx' \otimes hh') = xg' \otimes hx' + gx' \otimes xh',
\end{gather*}
	where $xg'$ is a $(gg',hg')$-skew primitive, $hx'$ is a $(hg',hh')$-skew primitive, $gx'$ is a $(gg',gh')$-skew primitive and $xh'$ is a $(gh',hh')$-skew primitive. Therefore,
\begin{gather*}
\Delta(xx') - (gg' \otimes xx' + xx' \otimes hh') \in \FBQ_1C_{gg',hg'}
 \boxtimes \FBQ_1C_{hg',hh'}
 \\ \qquad
 {}+ \FBQ_1C_{gg',gh'} \boxtimes \FBQ_1C_{gh',hh'}
 \end{gather*}
	and $xx' \in \FBQ_2C_{gg',hh'}$.
	We claim any word of skew primitive elements $x_n\cdots x_1$ belongs to a component. We proceed by induction with the above observation
providing the base case. Suppose $x_1\cdots x_{n-1}$ belongs to some component for every possible product of skew primitive elements.
Applying Lemma \ref{lem:skewfilter} to $x_n(x_{n-1}\cdots x_1)$ finishes that proof.
\end{proof}

\begin{Lemma}
\label{lem:skewfilter}
	Let $B$ be a bialgebra and $g$, $h$, $p$, $q$ be grouplikes and $x$ be $(p,q)$-skew primitive, then
	$pC_{g,h}^Q \subset C_{pg,ph}^Q$ and $xC_{g,h}^Q \subset C_{pg,qh}^Q$.
\end{Lemma}

\begin{proof}
	In fact, $pF_{n}B_{g,h}^Q \subset F_{n}B_{pg,ph}^Q$ and $xF_{n}B_{g,h}^Q \subset F_{n+1}C_{pg,qh}^Q$. We prove these relations by induction.
	For the base case, we use the fact that the multiplication of two grouplike ele\-ments is grouplike and the multiplication of a grouplike element and a skew primitive ele\-ment is skew primitive. In~particular,
	we have $gx$ is a $(gp, gq)$-skew primitive ele\-ment.
	Now we assume $pF_{n-1}B_{g,h}^Q \subset F_{n-1}B_{pg,ph}^Q$
	Observe for $z \in F_nB^Q_{g,h}$ and $\overline{\Delta}_{g,h}(z)= \sum z_{(1)} \otimes z_{(2)} \in \sum_{t \in \SGL(C)} F_{n-1}B^Q_{g,t} \boxtimes F_{n-1}B^Q_{t,h}$.
	By algebraic manipulations, it follows
	$\overline{\Delta}_{pg,ph}(pz) = \sum pz_{(1)} \otimes pz_{(2)}$,
	where $\overline{\Delta}_{pg,ph}(pz) \in \sum_{t \in \SGL(C)} pF_{n-1}B^Q_{g,t} \boxtimes pF_{n-1}B^Q_{t,h}$.
	By induction hypothesis: $\overline{\Delta}_{pg,ph}(pz) \allowbreak\in \sum_{t \in \SGL(C)} F_{n-1}B^Q_{pg,pt} \boxtimes F_{n-1}B^Q_{pt,ph}$.
	Therefore, $pz \in F_nB^Q_{pg,ph}$ and $pF_{n}B_{g,h}^Q \subset F_{n}B_{pg,ph}^Q$.
	Now, we assume $xF_{n-1}B_{g,h}^Q \subset F_{n}C_{pg,qh}^Q$. Observe
\begin{gather*}
\overline{\Delta}_{pg,qh}(xz) = xg \otimes qz + \sum pz_{(1)} \otimes xz_{(2)} + \sum xz_{(1)} \otimes qz_{(2)} + pz \otimes xh,
\end{gather*}
	so
\begin{align*}
\overline{\Delta}_{pg,qh}(xz) \in {}&F_1B^Q_{pg,qg} \boxtimes qF_{n}B^Q_{g,h} +\sum_{t \in \SGL(C)} pF_{n-1}B^Q_{g,t} \boxtimes xF_{n-1}B^Q_{t,h}
\\
&+	\sum_{t \in \SGL(C)} xF_{n-1}B^Q_{g,t} \boxtimes qF_{n-1}B^Q_{t,h} +pF_{n}B^Q_{g,h} \boxtimes F_{1}B^Q_{ph,qh}.
\end{align*}
	Using the previous result for grouplike elements and that $M \subset M'$ and $N \subset N'$ as submodules, then $M \boxtimes N \subset M' \boxtimes N'$.
\begin{align*}
\overline{\Delta}_{pg,qh}(xz) \in {}&F_nB^Q_{pg,qg} \boxtimes F_{n}B^Q_{qg,qh} +\sum_{t \in \SGL(C)} F_{n}B^Q_{pg,pt} \boxtimes F_{n}B^Q_{xt,xh}
\\
& + \sum_{t \in \SGL(C)} F_{n}B^Q_{xg,xt} \boxtimes F_{n}B^Q_{qt,qh} +F_{n}B^Q_{pg,ph} \boxtimes F_{n}B^Q_{ph,qh}.
\end{align*}
	Hence, $xz \in F_{n+1}B^Q_{pg,qh}$.
\end{proof}

\begin{Proposition} \label{prop:bcoaction}
Let $B$ be a split sg-flavored bialgebra all of whose semigrouplikes are grouplike
and let $\pi\colon B\to B^{\rm red}$ be the projection.
The Brown map defined as
\begin{gather*}
\Delta_B\colon\quad B\to B\ot_R B^{\rm red}, \qquad \D_B(x)=\sum x^{(1)}\ot \pi\big(x^{(2)}\big)
\end{gather*}
is a Hopf coaction.
\end{Proposition}

\begin{proof}
The result of both iterations is $\sum x^{(1)}\ot \pi\big(x^{(2)}\big)\ot \pi\big(x^{(3)}\big)$.
\end{proof}

\subsection{Pathlike coalgebras}\label{sec:pathlike}
In many applications, we need a notion which is stronger than being sg-flavored, yet is weaker than being QT-connected.

\begin{Definition}
	A sg-flavored coalgebra is said to be \textit{pathlike} if
	\begin{enumerate}\itemsep=0pt
		\item $\SGL(C)=\GLE(C)$.
		\item $\FBQ_0C=\Free{\GLE}=\bigoplus_{g \in \GLE(C)}Rg$ as a submodule.
	\end{enumerate}
We say a pathlike coalgebra is split, if $F_0C$ is a direct summand, i.e., the sequence of $R$ modules $0\to F_0C\to C\to C/F_0C\to 0$ is split.

 A bi- or Hopf-algebra is pathlike or split pathlike if the underlying coalgebra is.	
\end{Definition}

 At the moment, we will disregard semigrouplike elements that are not grouplike.
 Handling such elements is more tricky and they only appear in torsion, but they might be of interest in the situation with modified
 counits, \cite[Section~2.2]{HopfPart2} and
 for applications like \cite{KaufmannMedina}.

 \begin{Example}
As expected by the name, the path coalgebra of a quiver $Q$ is pathlike. Indeed, $\SGL=\GLE$ equals the set of vertices $V(Q)$. The base is given by the sum of the $Rv\cap Rw=\delta_{v,w}R_v$, thus $\FBQ C=\bigoplus_{v\in V}Rv$.
 Note that $C_1C^Q_{g,h}=R\vec{e}+R(g-h)$, if there is a directed edge $\vec{e}$ from~$g$ to $h$. It is split and $\FBQ_0C$ is free.
 \end{Example}
 The assumptions allow us to simplify and strengthen Theorem~\ref{thm:sgflavor}:
\begin{Theorem}\label{mainthm2}
Let $C$ be pathlike coalgebra and $A$ algebra with $\mathop{\rm Ext}\big(C/\FBQ_0C,A\big)=0$. An element $f \in \Hom(C,A)$ has an $\star$-inverse in the convolution algebra $\Hom(C,A)$ if and only if for every grouplike element $g$, $f(g)$ has an inverse as ring element in $A$.

In particular, in this situation, a character is $\star$-invertible if and only if it is grouplike inver\-tible.
	
Furthermore, a pathlike bialgebra $B$ is a Hopf algebra if and only if the set of grouplike elements form a group.
\end{Theorem}

\begin{proof}
{\sloppy
By Theorem \ref{thm:sgflavor}, $f$ having a $\star$-inverse is equivalent to
 the restriction of $f$ on $\bigoplus_{g \in \GLE(C)}Rg$ considered as an element in $\Hom\big(\bigoplus_{g \in \GLE(C)}Rg,A\big)$ having a convolution inverse.
Say that for each $g\in \GLE(C)$, $f(g)\in A$ has an inverse then setting $h(g)=f(g)^{-1}$ defines a~convolution inverse on $\FBQ_0C$,
 as $(f\star h)(g)=f(g)h(g) = 1_A=\eta_A(\eps_C(g))$ --symmetrically for $h\star g$. On the other hand, if $f$ has a convolution inverse $f^{\star -1}$, then the extension of
 $h(g) =f(g)^{-1}$ as above also yields a convolution inverse and these must agree.

 }

	Using Corollary \ref{QTantipodecor}, and the theorem above, having an antipode on $B$ is equivalent to having an antipode on $F_0^QC=\bigoplus_{g\in \GLE} Rg$.
The only possible antipode on grouplike elements is \mbox{$S(g)=g^{-1}$} if the inverses exist. Hence the existence is equivalent to the existence of inverses.
\end{proof}

The following proposition will cover the targeted examples.
\begin{Proposition}
\label{prop:colorpath}
 A color connected colored coalgebra $C$ is a pathlike coalgebra. In~particular,
a $\bD$-flat color conilpotent colored coalgebra is a pathlike coalgebra.
\end{Proposition}
\begin{proof}
Let $C$ be colored by $X=\SGL(C)$. In~view of Proposition \ref{prop:coloredconil}, we need to check if the sg-flavored coalgebra is pathlike.
 Being color coaugmented means that all
semigrouplike elements are grouplike elements, since $(1-\eps(g))g=0$, and the submodule $Rg$ is free if and only if $\eps(g)=1$.
By definition $\FBQ_0(C)=C[\GLE(C)]=\bigoplus_{g\in \GLE(C)} Rg $.
\end{proof}

\begin{Proposition}
\label{prop:opprop}
The coopposite of an sg-flavored, respectively pathlike, coalgebra is also sg-fla\-vored, respectively pathlike.
\end{Proposition}

\begin{proof}
First, $\SGL(C^{\rm cop}) = \SGL(C)$.
Now, the definition of the filtration is symmetric and thus it is also exhaustive for $C^{\rm op}$. The further conditions for pathlike then only concern the $R$-module structure, which remains unchanged.
\end{proof}

\begin{Remark}
This means that there is a Drinfel'd double for sg-flavored or pathlike Hopf algebras, see, e.g., \cite[Section~IX.4]{Kassel}.
\end{Remark}

\begin{Corollary}
	The antipode of a split sg-flavored Hopf algebra $H$ is bijective.
\end{Corollary}
\begin{proof}
	 Using Proposition \ref{prop:opprop} and Theorem \ref{thm:sgflavor}, it follows that $H^{\rm cop}$ is also a Hopf algebra.
	 	 Indeed, $\id|_{\FBQ_0}$ is generated by $\SGL(C)$, so that the antipode on the generators exists and is fixed as $S(g)=g^{-1}$. When restricted to $\FBQ_0 H^{\rm cop}=\FBQ_0 H$, $S$ provides a convolution inverse for $\Delta^{\rm op}|_{\FBQ_0 H^{\rm cop}}=\Delta|_{\FBQ_0 H}$ and hence extends to $H^{\rm cop}$. Let $T$ be the antipode for $H^{\rm cop}$ and $S$ the antipode for $H$.
		
		 Then for $h \in H$, we have $h_{(2)}T(h_{(1)}) = \epsilon(h)1_H = T(h_{(2)})h_{(1)}$ (notice the difference from the normal axiom of an antipode). Applying $S$ and using the fact $S$ is an algebra antihomomorphism, it follows that
	 $(S \circ T)(h_{(1)})S(h_{(2)}) = \epsilon(h)1_H = S(h_{(1)})(S\circ T)(h_{(2)})$.
	 We see $(S \circ T)$ is both a left and a right convolution inverse to $S$. Thus, $S \circ T = {\rm Id}_H$ since the convolution inverse is unique. Replacing $h$ by $S(h)$ in $h_{(2)}T(h_{(1)}) = \epsilon(h)1_H = T(h_{(2)})h_{(1)}$, it follows that
	 $S(h_{(1)})(T \circ S)(h_{(2)})= \epsilon(h)1_H = (T\circ S)(h_{(1)})S(h_{(2)})$
	 Thus, $T \circ S = {\rm Id}_H$, and we conclude that $S$ is a bijection.
\end{proof}

A subclass of examples related to the topology of loop spaces~\cite{MayMoreConcise} is of the following special diagonal form.

\begin{Definition}
	A coaugmented coalgebra $C$ over $R$ which has a decomposition into Quillen connected subcoalgebras,
$C=\bigoplus_{g\in \GLE(C)} C_g$ with $C_g=Rg \oplus\overline{C_g}$, where $Rg$ is free
is called {\em looplike}.
It is a {\em component coalgebra} in case it is graded and the degree $0$ part is $C[X]$.
	 \end{Definition}

The name looplike stems from the fact a path coalgebra for a category that is totally disconnected has such a form. Disconnected means that for any two
	objects $X\neq Y$, $\Hom(X,Y)=\varnothing$ and hence all paths of composable morphisms are loops, viz.\ they have the same source and target.
		In the particular case of a quiver, the vertex set $V$ is the set of grouplike elements. In~the arrow set of $Q$ there is no arrow between any distinct vertices, i.e., $Q(x, y) = \varnothing$ if $x \neq y \in V$. Note that counterintuitively, even for a looplike coalgebra, $\FQT_{g,h}$ is non-zero as $g - h$ is a $(g, h)$-skew primitive, as mentioned above, see Remark~\ref{linermk}.

\begin{Remark}This notion of component coalgebra is motivated by $H_*(X, R)$ and $H_*(\Omega X, R)$,
where $X$ is a based space. 	If both modules are $R$-flat, the first homology ring is a component coalgebra and the second ring is a
 component Hopf algebra. Moreover, the condition for the second ring to be connected in the coalgebra sense is equivalent
 to $X$ being connected in the topology sense; cf.~\cite{MayMoreConcise}.
\end{Remark}

\begin{Proposition}A looplike coalgebra is a pathlike coalgebra.
\end{Proposition}
\begin{proof}
By definition $\bD_g(c)=\D(c)-g\ot c-c\ot g$ which when restricted to $C_g$
coincides with the comultiplication on the coaugmented algebra $C_g=Rg\oplus \overline{C_g}$.
The latter has $\eps_g(g)=1$, $\eta_g(1)=g$ and $\eta_g\eps_g=\phi_g$ as the projection to the first factor.
 Now as $\D(C_g)\subset C_g\bot C_g$ by assumption and hence $\bD_g\colon C_g\to C_g\bot C_{g,g}$ and $C_g\subset \FQ C_g$,
 since $C_g$ is Quillen connected. The counit is compatible by Lemma \ref{lem:associator}; viz.\ $\eps|_{C_{g}\subset \FQ C_{g,g}}=\eps_g$.
Thus, we can identify $C_g$ as lying inside the diagonal part $\FQT C_{g,g}$ of the bivariate Quillen filtration.
Hence the bivariate Quillen filtration is exhaustive as the $C_g$ already exhaust $C$. The degree $0$ part is $C[X]$ by definition
and there are no semigrouplike elements as $\eps(g)\neq 0$ for semigrouplikes.
\end{proof}

\section{Renormalization, quotients and localization}\label{renompar}
The obstruction to having a Hopf algebra structure on a pathlike bialgebra is the invertibility of the grouplike elements.
There are basically two approaches to remedy this perceived deficiency. One is adding formal inverses, which is possible by universal constructions.
But, as at the end of the day, for renormalization \`a la Connes--Kreimer, one actually only needs convolution
inverses for characters, there is another option. This is to restrict the target algebras or to place restrictions on the characters -- for instance that
 the target algebra is commutative, or that the characters restricted to grouplike elements have special properties. In~this case there are universal quotients that the characters factor through.
The commutativity assumption is natural. Namely, if the target algebra is not commutative, then the convolution of two characters need not be a character.
We first briefly explain the setup to give the motivation for these constructions.

\subsection{Recollections on renormalization via characters}
\label{sec:renom}
A renormalization scheme in the Connes--Kreimer formalization of the BPHZ renormalization~\cite{CK1}, is defined on the convolution algebra of a bialgebra with a Rota--Baxter (RB) algebra, see also \cite{Kurush, KurushKreimer, Guo, MarcolliNi}. Readers not familiar with RB algebras may consult Appendix \ref{rbapp}.

A Feynman rule is a character in the convolution algebra $\phi\in \Hom_{R\text{-}{\rm alg}}(B,A)$. The renormalization of $\phi$ based on a Rota--Baxter (RB)-operator is a pair of characters
$\phi_{\pm}\in \Hom_{R\text{-}{\rm alg}}(B,A_\pm)$
such that $\phi=\phi_-^{-1}\star \phi_+$, see Appendix~\ref{rbapp} for the definition of the subalgebras~$A_\pm$.
The character~$\phi_+$ is then the renormalized Feynman rule.
It was shown in \cite{CK1,CK2} that there is a unique such decomposition for the Connes--Kreimer Hopf algebra of graphs and a commutative RB algebra, if $\phi$ is the exponential of an
infinitesimal character. The solution is then given by a recursive formula. This was generalized in \cite{EGK2} for $B$ any connected bialgebra and a character that is the exponential of an infinitesimal character. The following analysis provides the background for a~generalization of these results to the case of a bialgebra with a QT-sequence.
 If $B$ is a Hopf algebra, then the inverse is given by $\phi_-^{-1}=\phi_-\circ S$, see Proposition \ref{convolutionprop}, but if $A$ or the character has special properties, then the full assumption of $B$ being Hopf is not necessary.

\subsection[Quotients and quantum deformations for grouplike central and invertible characters]{Quotients and quantum deformations for grouplike central \\and invertible characters}\label{sec:quot}
We have already seen that taking the somewhat drastic quotient by $F_0^Q(C)$ makes an sg-flavored coalgebra connected.
There are, however, intermediary quotients, which make characters with special properties invertible.
These are also motivated by the natural examples \cite{CK,CK1,CK2,KurushKreimer,MarcolliNi}. We use the notation $1=1_B$.

\begin{Proposition}\label{commprop}
The ideal $J=([B,B])$ spanned by the commutators is a coideal. And,
if $A$ is commutative, then any character $\phi\in \Hom_{R\text{-}{\rm alg}}(B,A)$, factors through $B^{\rm com}=B/([B,B])$.
\end{Proposition}

\begin{proof}$J$ is a coideal: $\Delta(ab-ba)=\sum a_{(1)}b_{(1)}\ot a_{(2)}b_{(2)}-b_{(1)}a_{(1)}\ot b_{(2)}\ot a_{(1)}
=(a_{(1)}b_{(1)}-b_{(1)}a_{(1)})\ot a_{(2)}b_{(2)}+ b_{(1)}a_{(1)} (a_{(2)}b_{(2)}-b_{(2)}a_{(2)}) $ and $\eps(ab-ba)=0$.
Since $\phi\colon I\subset \ker(\phi)$, the statement follows.
\end{proof}

\begin{Proposition}\label{normalizedprop}
A grouplike normalized character factors through the quotient $B^{\rm red}=B/I_N$, where $I_N$ is the ideal spanned by $1-g$ for $g\in \GLE(B)$.
\end{Proposition}
\begin{proof} We need to check that $I_N$ is a coideal. Indeed, $\Delta(1-g)=1\ot 1-g\ot g=(1-g)\ot (1+g)$ and $\eps(1-g)=0$, since $g$ is grouplike.
Furthermore if $\phi$ is normalized the $\phi(1-g)=0$ so that $I\subset \ker(\phi)$.
\end{proof}
{\bf NB:} $R(g-h)\subset I$ since $g-h=(1-h)-(1-g)=g-h$. This is the line which keeps the sum in a pathlike bialgebra from being direct, see Remark~\ref{linermk}.

There is a further quotient that is of interest which was studied in \cite{HopfPart1}.
Consider the ideal~$I_C$ generated by $(ag-ga)$ for $a\in B$, $g\in \GLE(B)$.
\begin{Proposition}
$I_C$ is a coideal and the bialgebra $B/I_C$ and any grouplike central character factors through $B/I_C$.
\end{Proposition}

\begin{proof}
The fact that $I_C$ is a coideal follows as in Proposition \ref{commprop}. For any grouplike central character, $\phi(ag-ga)=0$, so $\ker(\phi)\supset I_C$.
\end{proof}

For a grouplike invertible central character, there is a universal construction which can be viewed as a quantum deformation and gives rise to a Brown type coaction.
Consider the algebra $B_\bq=B\big[q_g,q_g^{-1}\big]/K\colon g\in \GLE(B), g\neq 1$, that is the ring of Laurent polynomials with coefficients in the possibly noncommutative $B$ modulo the ideal $K$ generated by $q_gq_h=q_{gh}$, which is well defined as the grouplikes form a monoid.
The polynomial rings $B_q=B[q_g]/K$, $q \in \GLE(B)$, $g\neq 1$ is a subbialgebra. The bialgebra $B_q$ can be viewed as a multi-parameter quantum deformation.

Note the $q_g$ lie in the center of $B_\bq$. Endow $B_\bq$ with the bialgebra structure, where the $q_g$,~$q_g^{-1}$ are grouplike, then $K$ is a coideal.
Consider the ideal $I$ of the Laurent series generated by $q_g-g$ which descends to $B_\bq$. In~the quotient $B_\bq/I$, the image of the grouplike elements
lies in the center.

\begin{Proposition}
$I$ is a coideal for $B_\bq$ and similarly for its subbialgebra $B_q$.

{\sloppy
Any grouplike central character can be lifted to $B_q$ and factors through the quotient \mbox{$B_q/I\simeq B/I_C$}.

}

Any grouplike invertible central character $\phi$ can be lifted to a character $\hat\phi\colon B_\bq\to A$. Furthermore,
 $\hat\phi$ factors through the quotient $B_\bq/I$ as $\hat\phi=\bar\phi\circ \pi$ and
 embedding $B\to B_\bq\to B_\bq/I\colon \bar\phi(\pi(b))=\phi(b)$.

As $q_g\to 1\colon B_\bq\to B$ and $B_q/I\to B/I_N$.
\end{Proposition}

\begin{proof}
By definition $\eps(q_g-g)=1-1=0$ and $\Delta(q_g-g)=q_g\ot q_g-g\ot g=(q_g-g)\ot (q_g+g)$, so $I$ is a coideal in all cases.
The lift is given by $\hat\phi(q_g)=\phi(g)$ and consequentially for $B_\bq$ $\hat\phi\big(q_g^{-1}\big)=\phi(g)^{-1}$. The other statements follow from $\phi(g)=\hat\phi(q_g)$.
Note that the image of $g$ in $B_\bq/I$ lies in the center, since $q_g$ does, which shows the isomorphism $B_q/I\simeq B/I_C$ as exactly these grouplike
elements have been made central.

The statement about the limit of $B_\bq$ is clear for $B_q/I$; notice that as $q_g\to 1$ all the elements $g\sim 1$ which is the quotient by $I_N$.

The last statement is straightforward.
\end{proof}

\begin{Proposition}
Let $B$ be split pathlike
then $\B_\bq/I\simeq B^{\rm red}\otimes_R R\big[q_g,q_g^{-1}\big]/K$
and similarly for the polynomial subrings.
\end{Proposition}

\begin{proof}The correspondence is given by collecting the factors of $q_g$, respectively $q_g^{-1}$ on the right, as they are central.
\end{proof}

This quantum deformation also lets us make the coaction of Brown of Proposition~\ref{prop:bcoaction} more commutative and gives a nice interpretation in terms of polynomials and Laurent series.
\begin{Definition}
The central quotient Hopf coaction for a split sg-flavored bialgebra is given by
\begin{gather*}
\D_B\colon\quad B_\bq/I\to B_{\bq}/I\ot B^{\rm red}, \qquad
\D_B(x)=\sum x_{(1)}\ot \pi(x_{(2)}),
\end{gather*}
where now both factors are Hopf algebras and $\pi$ simply sets the factors $q_g=1$ on the left to $1$.
\end{Definition}

\begin{Remark}
For the polynomial version in the operad case, the isomorphism in the
 proposition was realized as the quotient by the ideal $(|-q)$ in \cite[Section~2.6]{HopfPart1}, see also Section~\ref{operadpar}.

\end{Remark}

\subsection{Adding formal inverses}
In general, one can formally invert the grouplike elements, this is a universal construction called the group completion if one is speaking of a monoid or the localization
if one is inverting a~multiplicative subset. In~the current $R$-linear context this becomes a subalgebra.

Given a subalgebra $S$ of an $R$ algebra $A$, the localization at $S$, $S^{-1}A$ is an object, together with a morphism $j\colon A\to S^{-1}A$ with the universal property that
if $f\colon A \to A'$ is an algebra homomorphism such that for all $s\in S$, $f(s)$ is invertible, then $f$ factors as $g\circ j$, with $g\colon S^{-1}A\allowbreak\to A'$.
There is an abstract construction in the general case, but it is not effective if there are no other conditions.
It can be realized as a quotient of the free algebra of $R$-linear alternating words in $A$ and $S$. Assuming that $A$ unital, set
\begin{gather*}
T_R(A\ot S^{\rm op})= R\oplus A\ot S^{\rm op} \oplus A\ot S^{\rm op} \ot A\ot S^{\rm op}\oplus \cdots,
\end{gather*}
which is an algebra by concatenation. The unit is $1_R$, where we identify $R\ot_R M=M$ for any $R$-module $M$.
Then $AS^{-1}=T_R(A\ot S^{\rm op})/I$, where $I$ is the ideal spanned by be following relations:
\begin{enumerate}\itemsep=0pt
{\samepage\item[$1)$] $1\ot \bar 1=1_R$, $\bar 1\ot 1=1_R$ which makes $1\otimes \bar 1$ a unit,
\item[$2)$] $s\ot \bar s=1\ot \bar 1$, which inverts the elements of $s$,

}
\item[$3)$] $a\ot \bar 1\ot a'\ot \bar s=aa'\ot \bar s$,
\item[$4)$] $a\ot \bar s\ot 1\ot \bar s'=a\ot \bar s\bar s'$,
\end{enumerate}
where we write $\bar s$ to indicate the factors of $S^{\rm op}$. In~this notation $\overline{s_1s_2}=\bar s_2\bar s_1$ by the definition of the opposite multiplication.
We let $\eps(a\ot \bar s):=\eps(a)\eps(s)^{-1}$.

If $A$ is also a bialgebra, $S$ is also semigrouplike, and $\eps(S)\subset R^\times$ then we define the comultiplication:
\begin{gather*}
\Delta(a\ot \bar s)=\sum \big(a\Sp\ot \bar s\big)\ot \big(a\Spp\ot \bar s\big).
\end{gather*}

This fixes the comultiplication on the free algebra via the bialgebra equation.
Defining the counit as $\eps(a\ot \bar s)=\eps(a)\eps(s)^{-1}$, defines the counit on the free algebra.

\begin{Proposition}
$I$ is a coideal and hence $AS^{-1}$ is a bialgebra. Similarly one can define a left localization $S^{-1}A$ by using $T_R(S^{\rm op}\ot A)$.
\end{Proposition}

\begin{proof} This is a straightforward calculation:
\begin{enumerate}\itemsep=0pt
\item $\Delta(1\ot \bar 1-1_R)= 1\ot \bar 1\ot 1\ot \bar 1-1_R\ot 1_R=(1\ot \bar 1-1_R)\ot (1\ot \bar 1+ 1_R) \in I\ot A$.
\item $\Delta(s\ot \bar s-1_R)=(s\ot \bar s)\ot (s\ot \bar s)-1_R=(s\ot \bar s+1_R)\ot (s\ot \bar s-1_R)\in A\ot I$.
\item $\Delta(a\ot \bar 1\ot a'\ot \bar s-aa'\ot \bar s)= \sum a_{(1)}\ot \bar 1 \ot \sum a'_{(1)}\ot \bar s \ot a_{(2)}\ot \bar 1\ot a'_{(2)} \ot \bar s- \sum a_{(1)}a'_{(1)}\ot \bar s\ot a_{(2)}a'_{(2)}\ot \bar s \equiv 0 \mod (I\ot A+A\ot I)$.
\item $\Delta(a\ot \bar s\ot 1\ot \bar s'-a\ot \bar s\bar s')=\sum a_{(1)}\ot \bar s\ot 1\ot \bar s' \ot a_{(2)}\ot \bar s\ot 1\ot \bar s'-a_{(1)}\ot \bar s \bar s'\ot a_{2}\ot \bar s\bar s'\equiv 0 \mod (I\ot A+A\ot I). $
\end{enumerate}
For the counit, it is a straightforward check that $\eps(I)\equiv 0$.
\end{proof}

\subsection{Calculus of fractions} There is a more concise version of this construction if the Ore condition and cancellability are
met. This gives a calculus of fraction, similar to the calculus of roofs \cite{GelfandManin}.
The right Ore condition states that
\begin{gather*}
\forall\ a\in A,\ s\in S,\quad \exists\ a'\in A,\ s'\in S\colon\quad as=s'a'.
\end{gather*}
One also needs cancellability, or $S$-regularity. That is if $as=0$ or $sa=0$ then $a=0$. This means that $sa=sb$ implies $a=b$ as does $as=bs$.

\begin{Proposition}
If these conditions are met any element of $AS^{-1}$ can be written as $as^{-1}$. Furthermore, there is an injective algebra homomorphisms $j\colon A\to AS^{-1}$,
which is universal in the sense that any morphisms $f\colon A\to B$ such that $f(s)$ is invertible has a unique lift to $g\colon AS^{-1}\to B$ such that $f=gj$.
Finally, the left and right fractions become isomorphic.
\end{Proposition}
\begin{proof}
See \cite{Artin}.
\end{proof}

\begin{Remark}
 For a free non-commutative algebra $R\{x,y\}$ and $S=\langle y\rangle$ the condition is not met.
 For the Weyl algebra $R\{x,y\}/(xy-yx=1)$ the condition is met \cite{Artin}.

 In similar spirit, if the product is the tensor product in a monoidal category, the Ore condition is usually not met, as $X\ot Y$ is rarely $Y\otimes X$. In~the case of interest, where one needs to invert the identities $\id_X$, one would need a $Y$ and a morphism $g$ such that for a given morphism $f\ot \id_X=\id_Y\ot g$. In~the commutative case this always holds. In~particular, it does hold in the symmetric monoidal setting when using the coinvariants $\Biso$, see Section~\ref{channelpar}.

A further route of exploration is to use higher homotopy commutativity, i.e., an $E_1$ for the Ore condition.
 \end{Remark}

For a grouplike central character of a pathlike coalgebra, one does not need full commutativity of $B$ for the Ore condition, but only that of grouplikes.
 In this case, one can invert the ideal in $B/I_C$, where the Ore condition holds. Equivalently one can work with $B_q$ and the quotient by~$I$.
\begin{Lemma}
\label {Laurentlem}
In the special case of $B_q$ the multiplicative set $S$ spanned by the $q_g$ is Ore, since these elements are central, and $S$ is cancellable.
The localization $B_qS^{-1}$ is the Laurent-series ring $B_\bq$.
\end{Lemma}
\begin{proof}
The cancellability is clear, since a polynomial vanishes if and only if all of its coefficients do.
The isomorphisms is given by $\bar q_g\to q_g^{-1}$ which exists by the universal property.
\end{proof}

\begin{Remark}[Abelian case, Grothendieck construction]
In the fully commutative case, this concretely boils down to the Grothendieck construction and localization.
Given a monoid $(M,\cdot)$, the Grothendieck construction gives the group completion: $K(M)=M\times M/{\sim}$ where $(m,n)\sim(m',n')$ if $m\cdot n'=m'\cdot n$.
There is an injection $M\to K(M)$ given by $m\mapsto (m,1)$, with inverses given by $(0,m)$.
More generally, if $R$ is a commutative ring and $S$ a multiplicatively closed subset, the localization $S^{-1}$ is given by $S^{-1}R=R\times S/{\sim}$
where $(r,s)\sim (r',s')$ if there is a~$t\in S$ such that $t(rs'-r's)=0$.
Similar constructions can be found in~\cite{Schmitt}.
\end{Remark}

\begin{Remark}[braided/crossed case]
The Ore condition is clearly met when $S$ is central as exploited above. It is more lax, though. It for instance allows for a crossed product types and essentially formalizes bicrossed products. These appear naturally for isomorphisms \cite{LodayCrossed}, \cite[Section~6.2]{feynman}.
For the example of the algebra of morphisms of a category, the Ore condition is guaranteed if there is a commutative diagram for every pair $(a,s)$:
\begin{equation*}
\xymatrix{
X\ar[r]^{s}\ar[d]_{a'}&Y\ar[d]^{a}\\
X'\ar[r]_{s'}&Y'.
}
\end{equation*}
\end{Remark}

\subsection{Application to pathlike bialgebras}
The following results follow in a straightforward fashion from the previous ones.
\begin{Theorem}
\label{thm:mainpathlike}
 Let $B$ be a pathlike bialgebra:
\begin{enumerate}\itemsep=0pt
\item[$1.$] Every grouplike invertible character has a $\star$-inverse and vice-versa.
\item[$2.$] The quotient bialgebra $B/I_N$ of Proposition $\ref{normalizedprop}$ is connected and hence Hopf. In~particular, every grouplike normalized character $\phi\in \Hom(B,A)$ when factored through to $\bar\phi\in \Hom(B/I_N,A)$ has an inverse computed by $\bar \phi^{-1}=\bar\phi\circ S$.
\item[$3.$] The bialgebra $B_\bq/I$ is Hopf and the $\star$ inverse of a grouplike central character $\phi\!\in\! \Hom(B,A)$ when lifted and factored through $\overline{\hat\phi}\in \Hom(B_\bq/I,A)$ has an inverse computed by $\overline{\hat\phi}^{-1}=\overline{\hat\phi}\circ S$.
\item[$4.$] Let $S=\GLE(B)$, then the left localization $BS^{-1}$ is a Hopf algebra. Moreover if $S$ is satisfies the Ore condition and is cancelable, there is an injection $B\to BS^{-1}$ and $B_q/I=BS^{-1}/J$, where $J$ is the ideal generated by $bg-gb$ with $b\in B$, $g\in \GLE(B)$.
\end{enumerate}
\end{Theorem}

\begin{proof}
The first part is in Theorem \ref{mainthm2}. The second part is an application of the same theo\-rem. The third part then follows from Lemma \ref{Laurentlem} and Proposition \ref{prop:flavorgen}.
The last part is straightforward.
\end{proof}

\section{Colored bialgebras from Feynman categories}
\label{catpar}

Deconcatenation is a way to turn a category or colored monoid into a coalgebra, if certain conditions are met.
One way to obtain a bialgebra from a coalgebra $C$ is to consider the free algebra which is the tensor algebra $TC$ and
extend $\D$ using the bialgebra equation. More generally, if the colored monoid has a product structure, that is actually
a monoidal structure for the category, one can ask if the bialgebra equation holds for the product $\mu=\otimes$
 and the deconcatenation coproduct $\D$. The answer is that it does for Feynman categories, see \cite{feynman},
 that are suitably decomposition finite, so that $\D$ is well defined.
The axioms for a Feynman category roughly say that the objects
are a free $\ot$-monoid on basic objects and that the morphisms are a~free monoid on the so-called basic morphisms, which have $\ot$ indecomposables as targets.
There are two cases: one with symmetries, i.e., a symmetric monoidal structure, and one without.
The latter is simpler as the former necessitates to pass to equivalence classes, which also makes the product commutative.
Concretely, we will consider the bialgebras $B(\F)$ for non-symmetric and $\Biso(\F)$ for symmetric Feynman categories.
With some conditions, both cases are color coaugmented and hence pathlike by Proposition \ref{prop:colorpath} and hence Theorem \ref{thm:mainpathlike} applies.

The grouplike elements correspond to the objects and are represented by their identities $\id_X$ in the absence of isomorphisms and by their classes for $\Biso$.
Note these are usually not $\ot$-invertible. For this they would have to lie in $\operatorname{Pic}(\F)$,
which is not usually the case. It is, however, a natural assumption, that their images under characters are invertible as this happens in concrete applications.

\subsection{Feynman categories and gradings}

\begin{Notation}
 $\operatorname{Iso}(\F)$ is the subcategory of all objects of $\F$ with their isomorphisms.
 For a~category $\V^\ot$ will denote the free (symmetric) monoidal category. This is essentially the category of words (with or without symmetries) in $\V$, cf.\ \cite[Section~2.4]{matrix} for an overview and \cite{JoyalStreet, MacLane} for details. It comes equipped with a functor $j\colon \V\to \C^\ot$ and the universal property that any functor $F\colon \V\to\C$ to a (symmetric) monoidal category lifts to $F^{\ot}\colon \V^{\ot}\to \C$, viz.\ $F=F^{\ot} j$.

The comma category $(F\downarrow G)$ for two functors $F\colon \mathcal{D} \to \C$ and $G\colon \E\to \C$
has as objects $(X,Y,\phi)$ with $X\in \mathcal{D}$, $Y\in \C$, and $\phi\in \Hom_\C(F(X)\to G(Y))$.
The morphisms are given by commutative diagrams.
That is a morphism from $(X,Y,\phi)\to (X',Y,\phi')$ is given by a pair $f\colon X\to X'$, $g\colon Y\to Y'$, such that $G(g) \phi=\phi'F(f)$.
If $F,G$ are clear, we will write $(\mathcal{D}\downarrow \E)$ for the comma category.
The category $\operatorname{Iso}(\operatorname{ar}(F))$, viz.\ the underlying groupoid of the arrow category, has the morphisms of $\F$ as objects while the morphisms of the category are given by pairs of isomorphisms $(\s,\s')$ which map $f\mapsto \s^{-1}f\s'$ when the compositions are defined.

A slice category is the comma category
$(\C \downarrow X)$, where $(X,\id_x)$ is a viewed as a subcategory with one object.
\end{Notation}

\begin{Definition}
A Feynman category is a triple $(\V,\F,\imath)$, where $\V$ is groupoid, $\F$ is
a (symmetric) monoidal category and $\imath\colon\V\to \F$ is a functor, such that
\begin{enumerate}\itemsep=0pt
\item The free (symmetric) monoidal category $\V^\ot$ is equivalent to $\operatorname{Iso}(\F)$ via $\imath^\ot$.
\item The groupoid of morphisms $\operatorname{Iso}(\operatorname{ar}(\F))$ is equivalent to the free (symmetric) monoidal category on the morphisms from objects from $\V^\ot$ to $\V$. In~formulas $\operatorname{Iso}(\operatorname{ar}(\F))\simeq $ \mbox{$\operatorname{Iso}(\V^\ot\downarrow \V)^\ot$}.
\item All slice categories $\F$ are essentially small.
\end{enumerate}

A Feynman category is {\em strict} if its monoidal structure is strict.
\end{Definition}

The objects of $\V$, sometimes called vertices, are the basic objects in the following sense.
The first condition says that every object of $\F$ decomposes as $X\simeq \imath(*_1)\odo \imath(*_n)$ essentially uniquely, where this means up to isomorphisms on the indecomposables and in the symmetric case up to permutations.

 Dropping the $\imath$, the second condition means that any morphism $\phi\colon Y\to X$ of $\F$ decomposes into indecomposables as well $\phi=\phi_1\odo \phi_n$, where each of the $\phi_i\colon Y_i\to *_i$ is a basic morphism, with $X\simeq *_1\odo *_n$ and
$Y\simeq Y_1\odo Y_n$. Again this is essentially unique, that is up to isomorphisms on the $\phi_i$ and permutations in the symmetric case. The third condition is technical and used for the existence of certain colimits.

In general there is a natural grading given by $|X|=$ word length. That is if $X\simeq *_1\odo *_n$ then $|X|=n$. This is well defined due to axiom (i).
This give a grading on the morphisms as $|\phi|=|X|-|Y|$. All the algebras/coalgebra structures are graded by $|\cdot |$.

 \begin{Definition} A {\em degree function} on a Feynman category $\FF$ is a map $\text{deg}\colon \Mor(\F)\to \N_0$, such that:
$\text{deg}(\phi\circ\psi)=\text{deg}(\phi)+\text{deg}(\psi)$ and
 $\text{deg}(\phi\ot\psi)=\text{deg}(\phi)+\text{deg}(\psi)$, and
every morphism is generated under composition and monoidal product by those of degree $0$ and $1$.

 For a {\em weak degree function}, the first condition is relaxed to $\text{deg}(\phi\circ\psi)\geq \text{deg}(\phi)+\text{deg}(\psi)$.
In addition, a (weak) degree function is called {\em proper} if
 $\text{deg}(\phi)=0 $ if and only if $\phi$ is an isomorphism.
\end{Definition}

In practice, Feynman categories come with a presentation. This is a set of generators of the basic morphisms together with relations among them, cf.\ \cite[Section~5]{feynman}.
The generators are a set $P(\F)\subset (\F\downarrow \V)$ of non-invertible elements
such that any morphism in $(\F\downarrow \V)$ up to isomorphism can be written as $\hat \phi_n \hat \phi_{n-1}\cdots \hat \phi_2\phi_1$, where $\phi_i\in P(\F)$ and we use the short hand notation
 $\hat\phi=\id_{*_1}\odo \id_{*_{k-1}} \ot\phi\ot \id _{*_{k+1}}\odo \id_{*_n}$ with $\phi$ in the $k$-th position and the factors of $\id_{*_i}$ are identities of some of the basic objects.
 Note that generators $\phi_i\colon X_i\to \ast$ can have sources with $|X_i|\geq 0$.
 Commonly the sources can be restricted to $|X_i|\in \{0,1,2\}$ or even $\{1,2\}$. The relations are given by commutative diagrams in the $\hat \phi_k$.
 We call $\FF$ {\em primitively generated}, if it has a set of generators which are
 essentially irreducible under composition.

 We call a presentation {\em effective} if for each $\phi\in (\V\downarrow \F)$ there is a maximal number $\max{\phi}$ of generators in any decomposition as above.
Defining $\max{\phi_0\ot\phi_1}:=\max{\phi_1}+\max{\phi_2}$ extends $\max{{-}}$ to all morphisms.
The following is straightforward:
\begin{Lemma}
For a Feynman category with an effective presentation $\max{{-}}$ is a weak degree function.
A presentation in which the relations are homogenous in the number of generators is effective and $\max{{-}}$ is a degree function.
\end{Lemma}

\begin{Remark}
These considerations also work in an enriched setting, that is if the morphisms themselves are objects of a symmetric monoidal category, e.g., $R$-modules. We will not consider the details of this situation more deeply here and refer to \cite{feynmanrep}.
\end{Remark}

\subsection{Colored bialgebras from Feynman categories}

We recall \cite[Theorem 1.20]{HopfPart2}:
\begin{Theorem}
If a locally finite strict monoidal category $\F$ is part of a Feynman non-symmetric category $\FF$, then the algebra structure of $\ot$ and coalgebra structures of deconcatenation give a~bialgebra structure on $B(\F)$.
\end{Theorem}
This bialgebra is neither commutative nor cocommutative in general.
If $\F^{\rm op}$ satisfies the condition of the Theorem, we define $B(\F):=B(\F^{\rm op})^{\rm cop}$.

\begin{Theorem}
\label{thm:mainns}
With the assumptions as in the theorem above, $B$ is colored by its objects and color coaugmented.
If~$\FF$ additionally has a proper degree function, then it graded. If furthermore
 there are no isomorphisms except for the identities, it is color nilpotent and color connected.
 It~is then pathlike and Theorem~$\ref{thm:mainpathlike}$ applies.
\end{Theorem}

\begin{proof} That the coalgebra is colored and color coaugmented
follows from Lemma \ref{prop:catcolor}. The grading is clear.
The assertions in the case that there are no isomorphisms except the identities follow from Propositions \ref{prop:catcolor} and \ref{prop:colorpath}.
\end{proof}

 We now treat the case with isomorphisms. This will yield commutative but generally noncocommutative bialgebras. In~order to define the coproduct, some modifications are needed to avoid an over-counting of decompositions due to
the isomorphisms given by the permutations. Otherwise, the coproduct will not be coassociative, see \cite[Example 2.11]{feynman}.

\label{channelpar}

\begin{Definition}[\cite{HopfPart2}] A weak decomposition of a morphism $\phi$ is
a pair of morphisms $(\phi_0,\phi_1)$ for which there exist isomorphisms $\sigma$, $\sigma'$, $\sigma''$ such that
$\phi=\sigma\circ \phi_1\circ \sigma' \circ \phi_0 \circ \sigma''$. We introduce an equivalence relation which says that $(\phi_0,\phi_1)\sim (\psi_0,\psi_1)$ if they are weak decompositions of the same morphism.
An equivalence class of weak decompositions will be called a decomposition channel which is denoted by $[(\phi_0,\phi_1)]$.
$\F$ is called essentially decomposition finite, if the following sum is finite:
\begin{equation}
\label{Deltaiso}
 \Delta^{\rm iso}([\phi])=\sum_{[(\phi_0,\phi_1)]}[\phi_0] \otimes [\phi_1],
\end{equation}
where the sum is over a complete system of representatives for the decomposition channels for a fixed representative $\phi$.
\end{Definition}
Note that parallel to \eqref{eq:monoidal} this is the monoidal version of the coproduct. The categorical version is given by $\Delta^{\rm op}$.

If the category $\F$ is symmetric monoidal, there are extra symmetries permuting objects and morphisms, and the behavior of these automorphism groups is not bialgebraic. Therefore one has to take coinvariants under the action of the automorphism groups of the objects acting on the morphisms. Because the commutators and associators are isomorphisms,
the coinvariants are associative, commutative monoids. To take the coinvariants,
consider the $R$-module $R[\Mor(\F)]$ with $f\sim g$ if there are isomorphisms $\s$, $\s'$, s.t.\ $f=\s'f\s^{-1}$.
Note the product descends since if $\phi\sim \phi'$ and $\psi\sim \psi'$ then $\phi\ot \psi\sim \phi'\ot \psi'$.
We denote the quotient by $R[\Mor(\F)]^{\rm iso}=R[\Mor(\F)]/{\sim}$.

The following is contained in~\cite[Theorem 2.1.5]{feynman}:
\begin{Theorem}
\label{thm:isofey}
If a channel decomposition finite monoidal category $\F$ is part of a Feynman category $\FF$, then the algebra structure defined by $\ot$ together with the coalgebra structure given by \eqref{Deltaiso} define a
bialgebra structure on $R[\Mor(\F)]^{\rm iso}$.
The counit is given by $\eps([\id_X])=1$ and $\eps([\phi])=0$ if $[\phi]\neq [\id_{s(\phi)}]$.
\end{Theorem}

If $\F^{\rm op}$ satisfies the conditions of the Theorem, we define $\Biso(\F):=\Biso(\F^{\rm op})^{\rm cop}$.

Note that without changing the coalgebra, we can assume that $\F$ is strict and skeletal.

\begin{Lemma}
\label{lem:isocolor}
In the situation of the theorem above,
the semigrouplike elements are the classes of the identities $[\id_X]$ and they are grouplike.
\end{Lemma}

\begin{proof}
Consider the coproduct of any $\phi$ for which $[\phi]\neq [\id_{s(\phi)}]=[\id_{t(\phi)}]$ has at least two distinct terms
\begin{gather*}
\Diso([\phi])=[\id_{s(\phi)}]\ot [\phi]+[\phi]\ot [\id_{t(\phi)}] +\cdots.
\end{gather*}
These are the only two terms precisely if $[\phi]$ is skew-primitive.
No such $[\phi]$ is semigrouplike, but,
if $[\phi]=[\id_X]$, then these two decomposition channels coincide and $\D([\id_X])=[\id_X]\ot [\id_X]$
as there are no left or right invertible elements
by Lemma \ref{prop:catcolor}.
Since $\eps([\phi])=1$, these are grouplike.
\end{proof}

Using $\Delta^{\rm iso}$ thus turns essentially indecomposable morphisms into skew primitives.

\begin{Remark}
For a proper (weak) degree function the degree $0$ part is free on the grouplikes, which are the classes $[\id_x]$.
The degree $1$ part is generated by the isomorphisms classes of essentially indecomposable elements, whose classes are skew-primitive. If $\F$ is primitively generated, then the classes of the generators are in degree $1$.
\end{Remark}

\begin{Theorem}
\label{thm:mainiso}
If $\F$ satisfies the assumptions of Theorem $\ref{thm:isofey}$,
the bialgebra $\Biso(\F)$ is colored, color coaugmented, color conilpotent and color connected and hence pathlike and
Theorem $\ref{thm:mainpathlike}$ applies.

 If $\FF$ additionally has a proper degree function, then it is graded, color nilpotent and color connected.
Thus, it is pathlike and Theorem $\ref{thm:mainpathlike}$ applies.
\end{Theorem}

 \begin{proof}

 The colors are the isomorphism classes of objects which are grouplikes by Lemma \ref{lem:isocolor}.
 It is freely generated as a module by the isomorphism classes of morphisms, so the grouplike elements split off as a direct summand $C[R]=\bigoplus_{x\in R}[\id_{x}]$, where $R$ is a set of isomorphism classes of objects which can be identified with the set of objects of the skeleton of $\F$.

The left and right coactions are given by $\lambda([\phi])=[\id_{s(\phi)}]\ot [\phi]$ and $\rho([\phi])=[\phi]\ot [\id_{t(\phi)}]$. This is well
defined as can easily be checked. By definition of a composition channel $s(\phi_1)\simeq t(\phi_0)$,
so that the decomposition is colored. The coaugmentation comes from the splitting.

Since isomorphisms have degree $0$, a proper degree function descends to the isomorphism classes: $\operatorname{deg}([\phi]):=\operatorname{deg}(\phi)$.
For the same reason the coproduct \eqref{Deltaiso} is graded. As all the isomorphisms are in the classes of the identities,
$\Biso$ is color reduced and hence conilpotent by Corollary \ref{cor:gradedcolornil}.
As $\Biso$ is a free $R$-module, Lemma \ref{lem:bdflat} applies and it is connected.
Proposition~\ref{prop:colorpath} guarantees that $\Biso$ is pathlike in this case.
\end{proof}

 \begin{Remark} In these cases the localization and the Laurent series will give Hopf algebras.
The~$q$ deformations will have parameters given by the isomorphism classes of basic objects. The~$[\id_x]$ are also the
classes that will be formally inverted. This can be thought of as extending~$\F$ to a~localized category $S^{-1}\F$
such that the $x$ live in the Picard subcategory $\operatorname{Pic}\big(S^{-1}\F\big)$.
 \end{Remark}

\begin{Remark}
The characters need not be graded, but what is common is that they map the additive degree to a multiplicative one as in $n\mapsto q^n$, where $q$ can be of degree $0$, which can also be achieved by shifting degrees or inverting elements in the target. One typical such $q$ would be~$\frac{1}{2\pi {\rm i}}$, see \cite{KreimerYeats}, or factors of $\zeta^{\mathfrak m}(2)$ in the framework of Brown, see \cite{BrownICM}.

\end{Remark}

\subsection{Details for the main examples}
We will give the details for three types of examples, which cover the main applications.

\subsubsection{Set based/simplicial examples}
The category $\FinSet$ is a Feynman category. There is only one basic object up to isomorphism, which is a one element set.
The basic degree function $|X|$ coincides with the cardinality. This is not a proper degree function, but it is proper when restricted to the Feynman subcategory of finite sets and surjections ${\rm FS}$. Similarly, $-|\phi|$ is a proper degree function for the Feynman subcategory of finite sets and injections ${\rm FI}$, cf.\ \cite{HopfPart2}.
The category $\FinSet_>$ of ordered finite sets with order preserving maps is a non-symmetric Feynman category, cf.~\cite{matrix}. As before, $|\phi|$ is a~degree function which is proper when restricted to the non-symmetric Feynman subcategory~${\rm OS}$ of ordered finite sets and surjections, and $-|\phi|$ is a~degree function for ${\rm OI}$, that is ordered sets and injections. The skeleton of $\FinSet_>$ is the augmented simplicial category $\Delta_+$ and the subcategories of surjections and injections restrict as $\Delta_+^{\rm surj}$ and $\Delta_+^{\rm inj}$. For a skeletal $\V$ there is only one basic object $[0]=\{0\}$. Its powers can be identifies with $[0]^{0}=\varnothing=:[-1]$ and $[0]^{\ot n}=[n-1]=\{0\kdk n-1\}$ -- using the traditional simplicial notation.

\begin{Proposition}
The bialgebras $B\big(\Delta_+^{\rm surj}\big)$, $B\big(\Delta_+^{\rm inj}\big)$,
$\Biso({\rm FI})$, $\Biso({\rm FS})$ as well as their coopposites are all colored connected pathlike bialgebras. The colors can be identified with $\N_0$.
The only grouplike elements are the $\id_{[n]}$ which are monoidally generated by $\id_\unit$ and $\id_{[0]}$, respectively their isomorphism classes.

In the case of surjection the $\sigma_i^n\colon [n+1]\to [n]$ are $([n+1],[n])$ skew primitive
$\mu=\sigma^1_1{:}$ $[1]\to [0]$, respectively its class, is a skew primitive generator. In~the case of injection the $\delta_i^n\colon [n-1]\to [n]$ are $([n-1],[n])$-skew primitive and
$\delta_1^0\colon \varnothing \to [0]$, respectively its class, is a~skew-primitive generator.

Any character $\chi$ for which $\chi_{\id_{[0]}}$ is invertible is invertible.
Normalized characters pass through the quotient $B^{\rm red}$. A character is normalized if and only if $\chi_{\id_{[0]}}=1$.

In the localized version $S(\mu)=-\id_{[1]}^{-1}\mu \id_{[0]}^{-1}$, as $\id_{[1]}^{-1}=\id_{[0]}^{-2}$ and in Laurent series
$S(\mu)=-q^{-3}\mu$.

\end{Proposition}
\begin{proof}
The first set of statements follows from Theorem \ref{thm:mainns} and Theorem \ref{thm:mainiso} together with Proposition \ref{prop:opprop}. The claims about generation are straightforward. As $\id_{[n]}=\id_{[0]}^{\ot n-1}$, for any character $\chi(\id_{[n]})=\chi(\id_{[0]})^{n-1}$ which proves the statements about the characters. The last statement follows from the definitions.
\end{proof}

\begin{Remark}
By Proposition \ref{prop:opprop} the coopposite bialgebras are pathlike, and
Joyal duality, cf.~\cite{joyal}, takes on the following form: $B\big(\D_+^{\rm surj}\big)^{\rm cop}=B\big(\big(\D_+^{\rm surj}\big)^{\rm op}\big)=B\big(\D^{\rm inj}_{*,*}\big)$, where $\D^{\rm inj}_{*.*}\subset \Delta_{*,*}$ is the subcategory of the double base point preserving injections.
It is this bialgebra that is used by Goncharov and Brown, cf. \cite[Section~3.6.5]{HopfPart2} for details. The condition that $\chi_{\id_[0]}=1$ which under Joyal duality reads $\chi(\id_{[1]})=1$ is implemented as $I(0,\varnothing,1)=1$ in \cite{Gont}, cf. \cite[equation~(1.5)]{HopfPart1}.
Any such character is group normalized.

Goncharov's symbols $\hat I(a_0:a_1\kdk a_n:a_{n+1})$ can be understood as formal characters for a~decoration of
 $\D^{\rm inj}_{*,*}$, see, e.g., \cite[Section~1.1.5]{HopfPart1}.
 The equation that
$\hat I(a:\varnothing:b)=0$ then means that these characters factor through the connected quotient given by the coideal $\id_\unit-\id_{[1]}$ in~$\Delta_{*.*}$
 Not insisting on this relation, one can actually assign any invertible value to $\hat I(a:\varnothing:b)$, which is then captured by
the Laurent series quotient and Brown's coaction.
 In fact, the symbols of Goncharov are characters with values in a commutative ring, so that one can reduce to $B^{\rm com}$.

Contrarily to this, the coproduct of Baues, which is also a decoration,
is not commutative. The normalization condition in this case is
given by insisting that the space be simply connected implemented by collapsing the 1-skeleton,
 see the discussion in \cite{HopfPart1}.

\end{Remark}
\begin{Remark}
Localizing $\id_{[0]}$, we obtain a Hopf algebra with an invertible antipode. This means that one can construct the Drinfel'd double,
which is then a sort of bicrossed product of surjections and double base point preserving injections.
This should correspond to the Reedy category on structure on the simplicial category. The precise analysis will appear in future work.
\end{Remark}

\subsection{(Co)operadic examples}
\label{operadpar}
The second type of Feynman category is an enrichment ${\rm FS}_\O$ of ${\rm FS}$ or $\D_{+\,\O}^{\rm surj}$ of $\D_+^{\rm surj}$ as defined in \cite{feynmanrep,feynman}. Note that we can enrich in any monidal category $\E$. For the current purposes this category should be $\Set$ or $R$-mod, where $R=\Hom(\unit_\E,\unit_\E)$ is commutative.

Switching to the notation $n$ for the set $\{1\kdk n\}$
and setting $\Hom_{{\rm FS}_{\O}}(n,1)=\O(n)$ the second axiom of a Feynman category yields that $\Hom_{{\rm FS}_\O}(n,m)=\O^{\rm nc}(n,m)$, where up to isomorphism:
\begin{gather*}
 \O^{\rm nc}(n,k)=\bigoplus_{\substack{(n_1, \dots,n_k)\\ \sum_{i=1}^k n_i =n}}
\O(n_1)\odo \O(n_k).
\end{gather*}
Here and in the following $\bigoplus$ stands for colimit and we will use $\otimes$ for the monoidal product in $\E$. If the $\O(n)$ are sets, then these are simply $\amalg$ and $\times$.
The basic compositions $\O^{\rm nc}(n,k)\ot \O^{\rm nc}(k,1)$ give maps
\begin{gather*}
\gamma_{k,n_1\kdk,n_k}\colon\ \O(k)\ot \O(n_1)\odo \O(n_k)\to \O\Big(\sum n_i=n\Big).
\end{gather*}
 In the non-symmetric case, this is the data for a non-symmetric operad and in the symmetric case, the automorphisms $\SS_n$
 act on the objects yielding a compatible action by pre- and post-composition
on the $\O^{\rm nc}(n,k)$ which is compatible with maps $\g$ and the data is that of a symmetric operad.
Vice-versa, specifying a non-symmetric or a symmetric operad fixes an
enrichment under the condition that $\O(1)$ splits as $\unit \oplus \bar \O(1)$, where $\unit$
is the component of $\id_1$ and $\bar \O(1)$ does not contain any invertible elements, \cite[Proposition 3.20]{feynmanrep}.
An operad is called reduced, if $\bar \O(1)$ vanishes.
The natural grading is $|\O^{\rm nc}(n,k)|=n-k$.

We will assume that $\O(0)=\varnothing$; if it is not, the considerations of \cite[Section~2.11]{HopfPart2} apply.

\begin{Proposition}
For a non-sigma operad $\O$, with split $\O(1)$ as above, $\D^{\rm surj}_{+\, \O}$ is decomposition finite if and only if
 $\O(1)$ is decomposition finite. In~this case, $B=B\big(\D^{\rm surj}_{+\, \O}\big)$ is a bialgebra.
 The basic grading induces a grading for $B$.

The elements $\id_n=\id_1^{\ot n}$ are the grouplike elements and there are no other semigrouplike elements.
$B(\D^{\rm surj}_{+\, \O})$ is a colored bialgebra, and it is color nilpotent precisely if $C[\O(1)]$ is. In~this case, it is color connected and pathlike if it is $\bD$-flat, e.g., if $\O$ is set valued.

A normalized character factors through $B^{\rm red}$ in which all objects are identified. The deformation $B_q$ has one deformation parameter $q=q_{\id_1}$
and a character $\chi$ is grouplike invertible if and only if $\chi(\id_1)$ is.
\end{Proposition}
\begin{proof}
The first assertion follows from the fact that the natural degree allows to reduce the question of
decomposition finiteness and
nilpotence to~$\O(1)$, cf.~\cite{HopfPart1} for more details.

The next statement follows from Theorem \ref{thm:mainns}.
The statement of being color nilpotent follows directly from the fact that the degree function is a grading for the coalgebra structure.
The final statements are then straightforward.
\end{proof}
Analogously one can prove:

\begin{Theorem}
For a symmetric operad $\O$ with split $\O(1)$ as above ${\rm FS}_{\O}$ is channel decomposition finite, if and only if $\O(1)$ is. In~this case, $\Biso({\rm FS}_\O)$ is a bialgebra. The basic grading gives a degree function.
The elements $[\id_n]=[\id_1]^{\ot n}$ are the grouplike elements and there are no other semigrouplike elements and
 $B({\rm FS}_\O))$ is a colored bialgebra. It is color nilpotent precisely if its restriction to the isomorphism classes of $\O(1)$ is.
 In this case, it is color connected and pathlike if is $\bD$-flat, which is the case if $\O$ is set valued.

A normalized character factors through $B^{\rm red}$ in which all objects are identified.
 The deformation $B_q$ has one deformation parameter $q=q_{[\id_1]}$ and a character $\chi$ is grouplike invertible if and only if $\chi([\id_1)]$ is.
 \end{Theorem}

\begin{Example}[{CK-tree bialgebras}]
Concretely consider the operad of rooted trees or the non-sigma operad of planar planted trees, where $\O(n)$ are the trees with $n$ leaves and the composition is given by gluing on leaves. Then the coproduct is the familiar Connes--Kreimer type coproduct \cite{CK,foissyCR1,del}:
\begin{equation}
\label{eq:CKcoprod}
\D([\t])=\sum_{\t_0\subset \t}[\t_0] \ot [\t\setminus \t_0],\qquad \text{respectively}\quad
\D(\t)=\sum_{\t_0\subset \t}\t_0 \ot \t\setminus \t_0,
\end{equation}
where $\t_0$ is a rooted subtree --~stump~-- including all its tails and $\t\setminus \t_0$ removes all the vertices and half edges of $\t_0$, viz,\ the set of cut-off branches, see~\cite{HopfPart1} for details. A~summand is depicted in Figure~\ref{treecoprodfig}.
The classes are the~$\SS_n$ coinvariants which can be thought of as unlabeled. In~the symmetric case the r.h.s.\ is symmetric, that is a commutative product.

\begin{figure}[h] \centering
 \includegraphics[width=.7\textwidth]{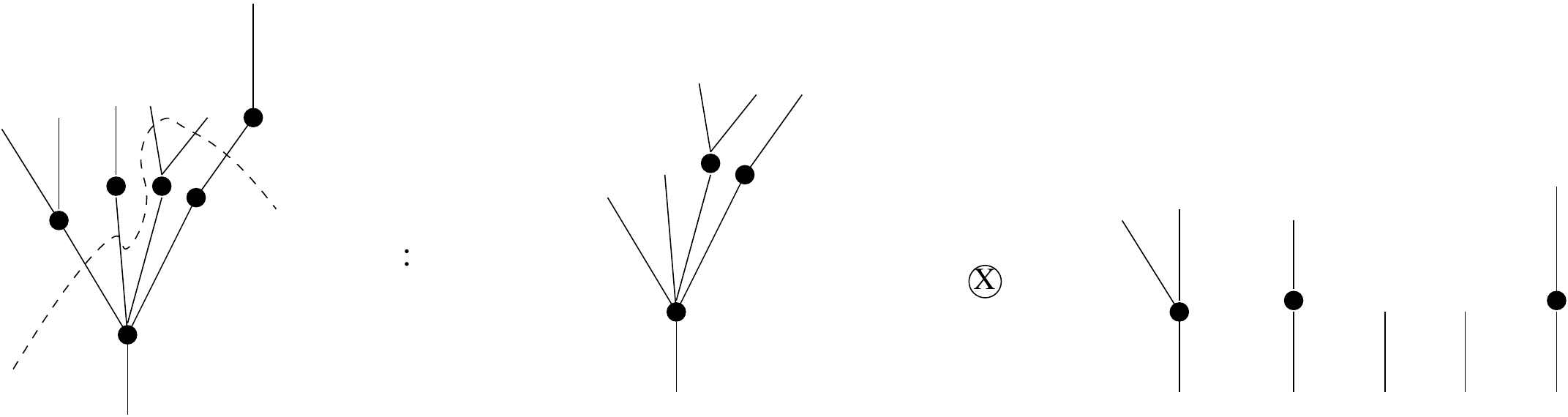}
 \caption{A summand of the coproduct for (planar) trees with leaves and roots. In~the planar case, the r.h.s.\ is either in the order depicted and in the non-planar it is a symmetric product.} \label{treecoprodfig}
\end{figure}

This Feynman category is primitively generated by the $\t_n$ which are the trees with $n$ leaves and one vertex. The relations are homogeneous amounting to associativity.
The corresponding degree function $\max{{-}}$ is proper and is equal to the number of vertices. Thus the bialgebras are graded and color connected.

 In this language a particular property of the CK-Hopf algebra becomes apparent. It is a~pathlike coalgebra whose grouplike elements are generated by $e=\id_\unit$ and $|=\id_{ 1}$ and it has a unique $\id_{{1}}$ primitive $|\hskip -4pt \bullet$ which generates a split line $R\,|\hskip -4pt \bullet$.
This yields the special properties explored in \cite{MoerdijkKreimer} and its generalizations, \cite[Example~2.50]{HopfPart1}, where, via coloring, many of these are introduced as additional direct components, so that the bialgebras are still pathlike.

In the localized Hopf algebra, $S(|\hskip -4pt \bullet)=-\id_{ 1}^{-1}\ot |\hskip -6pt \bullet\ot \id_{ 1}^{-1}$ and hence in the Laurent series $S(|\hskip -4pt \bullet)=- \,|\hskip -6pt \bullet q^{-2}$. More generally, for the skew primitive $\t_n$,
$S(\t_n)=-\id_{ 1}^{-1}\ot\t_n\ot \id_{ n}^{-1}$, respectively, $S(\t_n)=-q^{-(n+1)}\t_n$.
 The Brown coaction is given by setting the $q$ terms on the right in $\D$ to $1$.

This explains the fact that the infinitesimal structure can be given by a $q$-shifted comultiplication and a residue, see \cite[Corollary 2.2.5]{HopfPart1}.
 This generalizes Brown's infinitesimal structure to the level of operads.

\end{Example}

\subsection{Feynman categories of graphs}
There is a basic Feynman category of graphs called $\GG$, see \cite[Section~2.1]{feynman}.
We review salient features here. In~the appendix there is a purely set-theoretical version, which is of independent interest and
makes things even more combinatorial and hence is
conducive to pro\-gramming.\looseness=1

The category $\V$ has $S$-labeled corollas $*_S$ as objects with $\operatorname{Aut}(*_S)\simeq \operatorname{Aut}(S)$. This means
that a general object is a forest or aggregate of such corollas.
The connected version $\GG^{\rm ctd}$ has morphisms generated by simple edge contractions:
$\scirct\colon *_S\amalg *_T\to *_{S\setminus {\s} \amalg T\setminus\{t\}}$ and simple loop contractions
$\circ_{ss'}\colon *_{S}\to *_{S\setminus \{s,s'\}}$. These satisfy quadratic relations and are crossed with respect to the isomorphisms.
For any $\phi$ and any isomorphism $\sigma$ there are unique $\phi'$ and $\sigma'$
such that $\phi\circ \sigma =\sigma'\circ \phi'$.
Each morphism has an underlying graph $\gh(\phi)$ which is the source corolla, but with one edge for each operation $\scirct$ or $\circ_{st}$. The edges are
 glued from the half edges $s$ and $t$.
It~is defined in a way such that $\gh(\phi\circ \sigma)=\gh(\phi')$, in the notation above,
and its automorphisms are induced from those of the source, see \cite{DDecDennis,feynman}.
The isomorphism class of this graph gives the isomorphism class of $\phi$:
$[\gh(\phi)]=[\phi]$.

$\GG^{\rm ctd}\subset \GG$ is a sub-Feynman category whose basic morphisms have a connected ghost graph.
A morphism $\phi$ in $\GG^{\rm ctd}$ is basic if and only if $\gh(\phi)$ is connected;
see Figure \ref{basicmorphismsfig}.

\begin{figure}[h]\centering
\includegraphics[width=.7\textwidth]{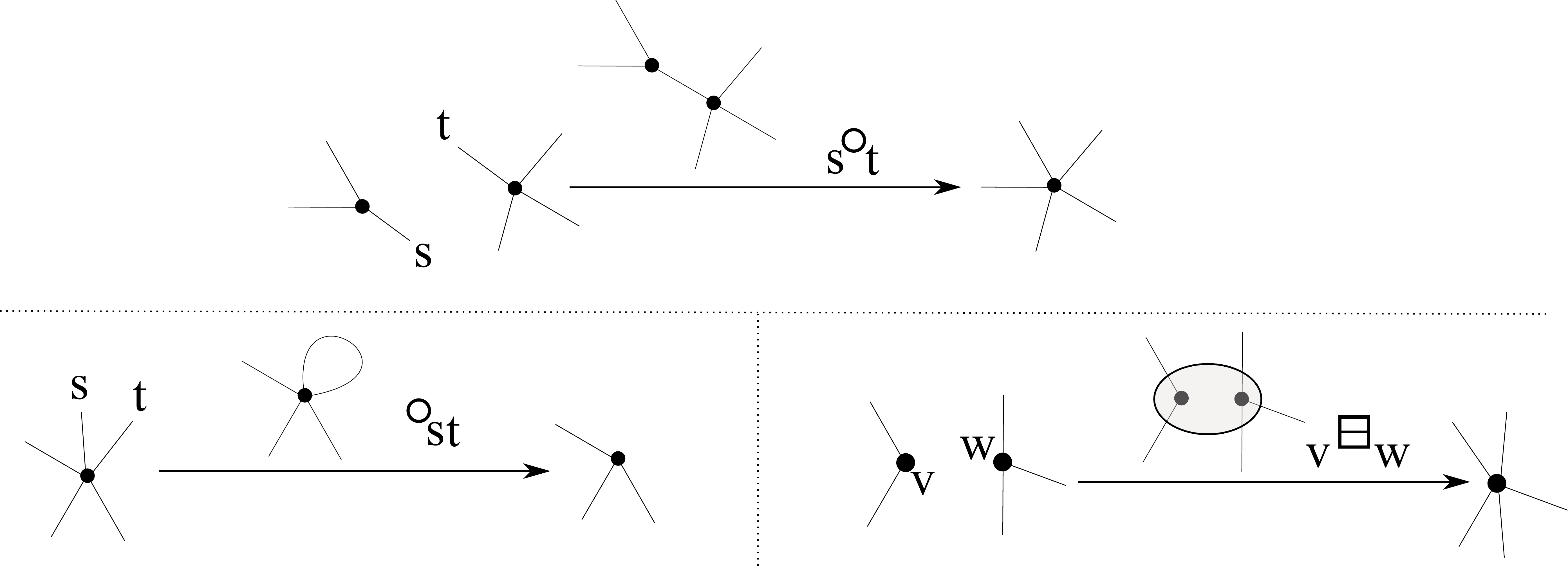}
\caption{The basic morphisms: Edge contraction loop contraction and merger, whose ghost graph is not connected. The shading is extra data not captured by the ghost graph.}\label{basicmorphismsfig}
\end{figure}

For the non-connected version, there is one more generator
$\mge{S}{T}\colon *_S\amalg *_T\to *_{S\amalg T}$. In~this case, there is a non-homogeneous relation
$\scirct=\circ_{s,t}\mge{S}{T}$ which allows one to replace the edge contractions and only retain loop contractions and mergers as generators.
The basic morphisms now need not to have connected underlying graph anymore.

Each corolla $*_S$ is isomorphic to $*_{ {|S|}}$ --$n$ is again short hand for the set $\{1\kdk n\}$~-- and hence $[\id_{*_S}]=[\id_{{|S|}}]$.
Furthermore fixing isomorphisms $\sigma\colon S\to {S}$, with $\sigma(s)=1$, $\sigma(s')=2$ and $\sigma'\colon T\to {T}$ with $\sigma'(t)=1$,
$[\scirct]=[\ccirc{0}{1}]$, $[\circ_{s,s'}]=[\circ_{1,2}]$ and
$[\mge{S}{T}]=\mge{{|S|}\,}{\, {|T|}}$.
Thus there is only one isomorphism class per basic morphism of generating type.

Beside the natural grading $|-|$, which is $1$ for an edge contraction or a merger and $0$ for an isomorphism or a loop contraction,
there is an additional degree given by
$\operatorname{deg}(*_S)=|S|$ which is additively continued $\operatorname{deg}(*_S\amalg *_T)=|S|+|T|$, to count the total number of legs in the forest.
This yields $\operatorname{deg}(\phi):=\frac{1}{2}(\operatorname{deg}(s(\phi))-\operatorname{deg}(t(\phi))$, which is integer, since both loop and edge contractions reduce the set of legs by an even number, while isomorphisms and mergers keep this number fixed. The degree of an isomorphism or a merger is hence $0$ and that of a simple loop or edge contraction is $1$. In~this way $\operatorname{deg}(\phi)=|E(\gh(\phi))|$.
The sum $wt(\phi)=\operatorname{deg}(\phi)+|\phi|$ is $1$ for the mergers and loop contractions and $2$ for the edge contractions.
This is also the degree $\max{{-}}$ for the presentation in terms of the $\scirct$, $\circ_{ss'}$ and $\mge{v}{w}$ and hence is a proper degree function.

\begin{Remark}[loop number and Euler characteristic]
In the graphical interpretation for a~basic morphism
$wt(\phi)=|E(\gh(\phi))|+b_0(\phi)$, where $b_0$ is the number of components of $\gh(\phi)$.
This is related to the Euler characteristic as follows $\chi(\gh(\phi))=wt(\phi)-|\phi|$,
where $\chi(\gh(\phi))=b_0(\gh(\phi))-b_1(\gh(\phi))=|V(\gh(\phi))|-|E(\gh(\phi))|$ and $b_1$ is the loop number, aka.\ first Betti number.\looseness=1
\end{Remark}

The coproduct in $\Biso$ is given by \cite[Section~3.6]{HopfPart2}.
\begin{gather}
\label{eq:graphcoprod}
\Delta{[\phi}]=\Delta[\gh(\phi)]=\sum_{\g\subset \gh(\phi)} [\g]\ot [\gh(\phi)/\g].
\end{gather}
This is a version of the core Hopf algebra \cite{Kreimercore}, but with legs and at a bialgebra level. See Figure~\ref{graphcoprodfig} for an example of a summand.

\begin{figure}[h]\centering
\includegraphics[width=.6\textwidth]{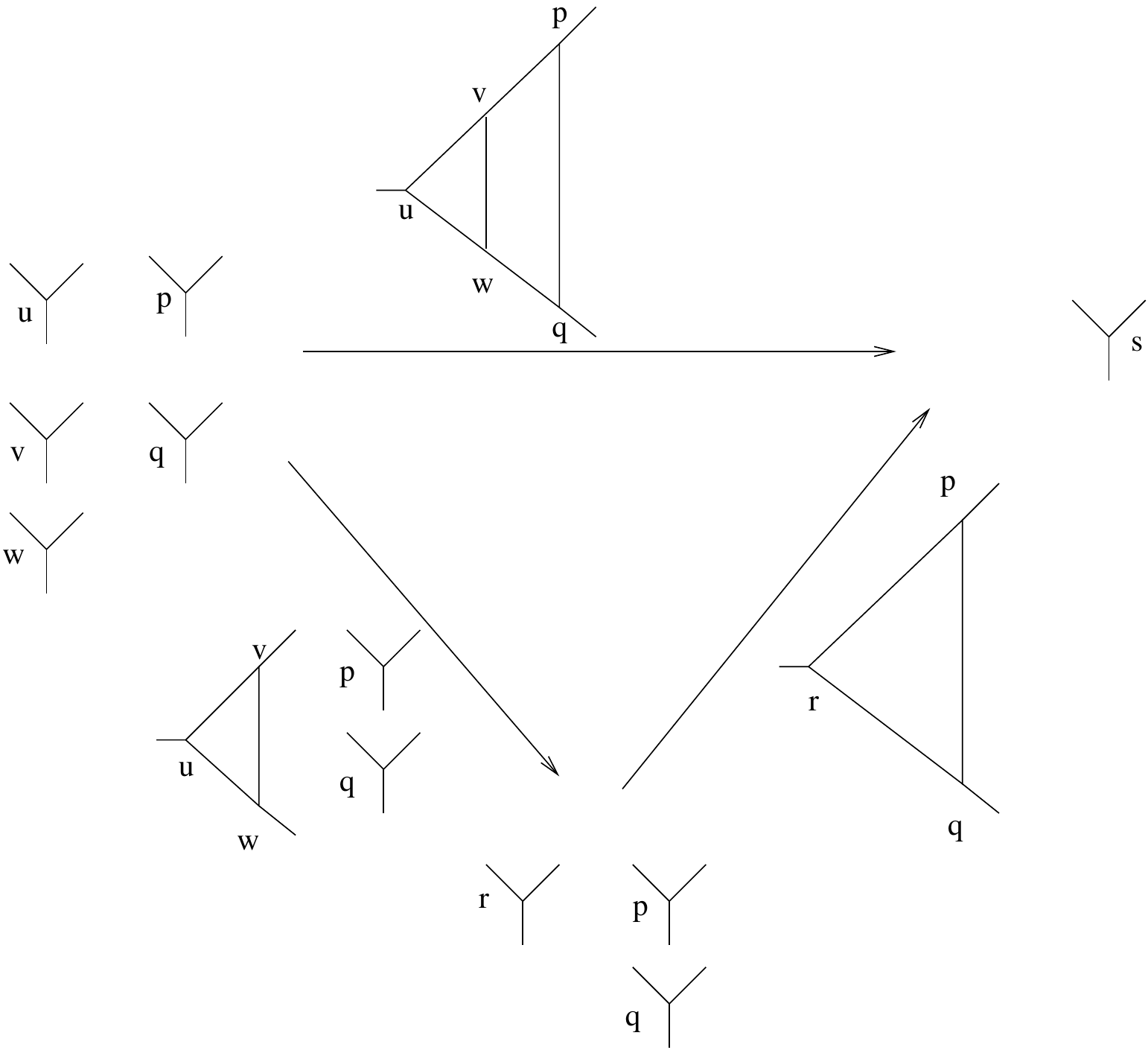}
\caption{A summand of the coproduct \protect{\eqref{eq:graphcoprod}} in terms of morphisms of aggregates. Note in the actual morphisms all flags and vertices are labeled. In~the figure only the vertices are tracked to avoid clutter. For the coproduct only the isomorphism classes of the morphism are involved.}\label{graphcoprodfig}
\end{figure}

\begin{Theorem}
Both $\Biso(\GG^{\rm ctd})$ and $\Biso(\GG)$ are color connected and hence pathlike. They are graded by $deg,| -|$ and $wt$, with $wt$ being a proper degree function on $\GG$ and $deg$ being proper on~$\GG^{\rm ctd}$.
For $\Biso(\GG^{\rm ctd})$ the bivariate Quillen index is given by $deg$ and for $\Biso(\GG)$ the bivariate Quillen index is given by $wt$.

More precisely, the grouplike elements are $[\id_\unit]$ and $[\id_{X}]$ which are generated by the classes $[\id_{*_S}]=[\id_{ {|S|}}]$. In~the connected case, the skew primitive elements are generated by the $[\scirct]$ and the $[\circ_{s,s'}]$ while in the connected case these
are generated by the $[\circ_{s,s'}]$ and the $[\mge{S}{T}]$ via products with grouplikes.

The antipodes for the skew primitives in the localization and the single parameter Laurent polynomials are given as follows:
{\samepage\begin{gather*}
S(\mge{n}{m})=-[\id_n]^{-1} [\id_m]^{-1}[\id_{n+m}]^{-1}\mge{n}{m}
S(\mge{n}{m})=-q^{-2(n+m)}\mge{n}{m},
\\
S([\scirct])=[\id_{|S|}]^{-1} [\id_{|T|} ]^{-1}[\id_{|S|+|T|-2}]^{-1}[\scirct]
S[\scirct]=-q^{-2(|S|+|T|+1)}\scirct,
\\
S([\circ_{s,s'}])=[\id_{|S|}]^{-1} [\id_{|T|}^{-1}] [\id_{|S|+|T|-2}]^{-1}[\circ_{s,s'}]
S[\scirct]=-q^{-2(|S|+|T|+1)}[\circ_{s,s'}].
\end{gather*}}

The connected Hopf quotient identifies all the objects with $\unit$ and hence all grouplikes $\id_X$ with $1=\id_\unit$.
The Brown coaction again sets the $q$ variables on the right to $1$, which halves the $q$ degrees in the above formulae.
\end{Theorem}

\begin{proof}
The first assertions follows from Theorem \ref{thm:mainiso} and the rest is straightforward from the definitions.
\end{proof}

\begin{Remark}
Besides the core Hopf algebra of all graphs, one can restrict and decorate as is done in \cite{brown,Kreimercore}.
For instance Figure \ref{graphcoprodfig} only contains 3-valent vertices, which corresponds to $\phi^3$ theory.
It is also possible to retract the legs, cf.\ \cite{HopfPart1,HopfPart2}.
\end{Remark}

\section{Conclusion and outlook}
\label{sec:conclusion}

\subsection{Conclusion}

We have shown that pathlike coalgebras in addition to quiver and incidence coalgebras capture the two types of bialgebras
stemming from Feynman categories, such as those of Baues, Goncharov and those of Connes--Kreimer. This realization provides several outcomes.

The structure of path-like coalgebras exactly answers the question of what the harbinger of the antipode in the various constructions
is. It precisely identifies the grouplike elements as the obstruction to the existence of the antipode.
For characters, their values on theses elements determines their $\star$-invertibility.
This also explains the appearance of quotients. These correspond to restrictions on the characters. The most common,
which says that all grouplike elements have character value 1, leads to the connected quotient. Other restrictions, like grouplike invertible
or grouplike central, yield different
quotients and explain the universality of the quantum deformation, which in turn explain the naturally of Brown's coaction.

This sets the stage for further work.
\subsection{Outlook}
\label{sec:drin}
Our analysis is based on using filtrations to reduce to an initial part, where the antipode or convolution inverse is known.
The notion of pathlike coalgebras codifies this situation if the initial part is setlike.
Section \ref{isopar}, however, can be taken as a clue that there could be some other structures,
which do not necessitate only grouplike elements, but allows for more information about the isomorphisms.
The following example of the Drinfel'd double leads the way to a new research in this direction.
Physically this corresponds to a (finite) gauge group. In~the infinite case, one has to take the precautions spelled out in Appendix \ref{sec:duals}.

 In particular, we present a double categorical interpretation in which the antipode basically switches the horizontal and vertical
 composition. The appearance of double categories is natural in view of
 \cite{DDecDennis,feynmanrep}. It opens the way to a different version of connectedness,
 whose lowest degree is of Drinfel'd double type rather than setlike, but keeps a link between the two approaches.
 We~will consider the generalization of the bialgebra and Hopf structures along these lines in \cite{KYM2}.
 There is a general construction, see \cite[Section~IX.4]{Kassel} of a Drinfel'd double for a finite dimensional Hopf algebra with invertible antipode; we will concentrate on the case $k[G]$. This is an interesting Hopf algebra whose algebra part comes from a category.

 \begin{Definition}
 For a finite group $G$ the {\it Drinfel'd double} $\dkg$ is the
quasi-triangular quasi-Hopf algebra whose
underlying {\em vector space has} the basis $\elt{g}{x}$ with $x,g\in G$
$\dkg=\bigoplus k \gx$ with bialgebra structure and antipode given by
\begin{gather*}
\gx\hy=\delta_{g,xhx^{-1}}\elt{g}{xy}, \qquad
\Delta(\gx)=\sum_{g_1 g_2=g} \elt{g_1}{x}\otimes\elt{g_2}{x}\,,\qquad
S(\gx)=\elt{x^{-1}g^{-1}x}{x^{-1}}\!.
\end{gather*}
The unit is $\sum_{g\in G}\elt{g}{e}$ and the counit is defined by $\eps(\elt{e}{x})=1$ and $\eps(\gx)=0$ if $g\neq e$.
\end{Definition}

 The algebra structure is the algebra of morphisms of the so-called loop groupoid, $\Lambda \underline{G}$. This is the category which has the object set $G$ and the
morphism set $\gx\in \Hom(g,x^{-1}gx)$, see, e.g.,~\cite{drin,Willerton}.
It is hence a colored monoid with colors $G$ with $s(\gx)=g$ and $t\big(\gx=x^{-1}gx\big)$.

 The coalgebra structure is not the dual, but the categorical coalgebra for the groupoid $\amalg_{x\in G}\underline{G}$ which has $X=G$ as objects and has
morphisms $\gx\colon x\to x$. Accordingly, the coalgebra is simply colored for the comodule structures for the set-like coalgebra $k[G]\colon \lambda(\gx)=x\ot \gx$, $\rho(\gx)=\gx\ot x$.

 The notation $\gx$ goes back to \cite{DPR}, where these are partition functions on a torus for a~field theory with finite gauge group. The two indices stand for the monodromies around the fundamental cycles, cf. \cite{orb,drin}.

\begin{Lemma}
The dual of $\dkg$ is a Hopf algebra, whose coproduct
is the coalgebra structure coming from the morphisms of a category. In~particular,
let $\delta_{\gx}$ be the dual basis then
\begin{gather*}
\d_{\gx}\d_{\hy}=\d_{x,y}\elt{gh}{x},\qquad
\Delta(\d_{\gx})=\sum_{y}\d_{\elt{g}{y}}\ot \d_{\elt{y^{-1}gy}{y^{-1}x}},\qquad
S(\d_{\gx})=\d_{\elt{xg^{-1}x^{-1}}{x^{-1}}}.
\end{gather*}
\end{Lemma}

\begin{proof}
Straightforward computation using the dual basis $\delta_{\gx}$.
 \end{proof}

The product structure uses the underlying monoid structure of the objects/colors. This is not quite a monoidal structure on the
category $\Lambda \underline{G}$. The definition of the antipode uses that the objects (as colors) are invertible. But note, that $\gx\ot \hy$ would need to be a morphism from $gh$ to $x^{-1}gxy^{-1}hy$ and if $x=y$ indeed
$\elt{gh}{x}$ is such a morphism. In~general, if $G$ is not Abelian no such morphism needs to exist.
The correct notion is that the
Drinfel'd double is a~subcategory of a natural double category structure on $\operatorname{Iso}(\C)$, see \cite[Appendix C]{feynmanrep}, where in this particular case $\C=\underline{G}$.
This is commensurate with the dictum: ``A monoidal category is a two category with one object''. The composability of horizontal 2-cells is then exactly the condition that $x=y$ in the above calculation. In~particular, the symbols $\gx$
are two-morphisms with the following horizontal and vertical composition:
\begin{equation*}
\xymatrix{
\ast\ar[r]^g\ar[d]_x \ar@{}[dr]|{\Downarrow \gx}&\ast\ar[d]^x\\
\ast\ar[r]_{xgx^{-1}}&\ast
}
\raisebox{-.6cm}{,}
\xymatrix{
\ast\ar[r]^g\ar[d]_x \ar@{}[dr]|{\Downarrow \gx} &\ast\ar[d]^{x}\\
\ast\ar[r]^{xgx^{-1}}\ar[d]_y \ar@{}[dr]|{\Downarrow \atop \elt{xgx^{-1}}{y}} &\ast\ar[d]^{y}\\
\ast\ar[r]_{xygy^{-1}x^{-1}}&\ast
}
\raisebox{-.6cm}{=}
\xymatrix{
\ast\ar[r]^g\ar[d]_{xy} \ar@{}[dr]|{\Downarrow \elt{g}{xy}}&\ast\ar[d]^{xy}\\
\ast\ar[r]_{xyg(xy)^{-1}}&\ast
}
\raisebox{-.6cm}{,}
\xymatrix{
\ast\ar[r]^g\ar[d]_x \ar@{}[dr]|{\Downarrow \gx}&\ast\ar[d]^x\ar[r]^h\ar@{}[dr]|{\Downarrow \elt{k}{x}}&\ast\ar[d]_x\\
\ast\ar[r]_{xgx^{-1}}&\ast\ar[r]_{hxx^{-1}}&\ast
}
\raisebox{-.6cm}{=}
\xymatrix{
\ast\ar[r]^{gh}\ar[d]_x \ar@{}[dr]|{\Downarrow \gx}&\ast\ar[d]^x\\
\ast\ar[r]_{xghx^{-1}}&\ast
}
\end{equation*}
The algebra structure is given by the categorical algebra for the vertical composition, viz.\ $\circ_h$,
as the colored multiplication with $s_v(\gx)=t_v(\gx)=x$:
\begin{gather*}
 \gx\cdot\hy=\delta_{s(\gx)=t(\hy)}\gx\circ_v\hy
\end{gather*}
and $\D$ will be given as the dual of the
horizontal composition with
$s_h(\gx)=g,t_h(\gx)=xgx^{-1}$
\begin{gather*}
 \D(\gx)=\sum_{(\hy,\kz)\colon \hy\circ_h\kz=\gx}\hy\ot\kz\,,
\end{gather*}
where we dropped the dual notation $\delta$.

 Another direction of further research is the connection to the cubical structure of the coproduct, whose analysis in this context appeared
in \cite[Section~4.5]{HopfPart1}. Cubical complexes more generally appear for cubical Feynman categories, \cite{Ddec}.
These are intimately related to Cutkosky rules \cite{Kreimercut} and the work \cite{KreimerYeats}.
To illustrate the cubical structure, we recall, cf.~\cite[equation~(4.29)]{HopfPart1}, that simplicial objects have a coproduct whose summands are
\begin{equation}
\label{simpcoprodeq}
x_{(0,i_1,\dots,i_{k-1},n)} \times
x_{(0,1,\dots,i_1)} x_{(i_1,i_1+1,\dots,i_2)} \cdots
x_{(i_{k-1},i_{k-1}+1,\dots,n)}.
\end{equation}
This has a cubical structure, which can be understood in terms of the Serre diagonal:
\begin{gather*}
\delta\colon\
 P=[0,1]^{n-1}\xrightarrow{\;\;\cong\;\;}
\!\!\!\! \!\!\!\! \!\!
\bigcup_{K\cup L=\{1,\dots,n-1\}}
\!\!\!\! \!\!\!\! \!\!
\partial_K^-[0,1]^{n-1} \times
\partial_L^+[0,1]^{n-1} \xrightarrow{\;\;\subset\;\;}
P\times P.
\end{gather*}
The term given in \eqref{simpcoprodeq} is identified with the subset $L=\{i_1,\dots,i_{k-1}\}$. In~view of \cite{KreimerYeats}
one has an additional interpretation. The right hand side
can be seen as cuts (or stops) in the path from~$0$ to~$n$ on the quiver $[n]=0\to 1\to \dots \to n$ and these can be viewed as cuts of a linear graph.
This is in line with the simplicial structure of iterated integrals.

Furthermore, there is a symmetry in the Serre diagonal of the cube given by interchan\-ging~$0$ and $1$ in the standard presentation. This switches the coproduct to its opposite. These two observations should generalize to the categorical setting, and when applied to graphs may explain some of the results of \cite{KreimerYeats}.
A last aspect is functoriality of the construction, see \cite[Section~1.7]{HopfPart2}, which is likely related the structure of cointeracting coalgebras \cite{Foissytalk,KreimerYeats}
in the present context.

Furthermore, we identified the skew primitive elements (as indecomposables) as being of central importance.
This parallels the role of primitives in the usual connected case. In~particular,
 there should be $B_+$ operators associated to them. These are indeed explicitly present and key elements for $E_2$-operads and hence
 for Gerstenhaber brackets and Deligne's conjecture~\cite{KZhang}. Moreover, there is a connection between these
 elements and the terms of the master equations~\cite{KMZ} and \cite[Section~7.5]{feynman}.
The general structure of $B_+$ operators will be taken up in~\cite{Bplus},
where we will also make the connection to Hochschild complexes.

\appendix
\section{Duals, actions and colored algebras}
 \label{sec:duals}
 \subsection{Duals}

The dual of an $R$ module $M^*$ is the $R$-module $\Hom_R(M,R)$. If $M$ is graded, $M=\bigoplus_{x\in X}M_x$ then
the graded dual is $\check{M}=\bigoplus_{x\in X} M_x^*$.

The dual of a coalgebra $C$ is the algebra $C^*:= \Hom(C,R)$ with product $f g(c)=(f\ot g)(\D c)$.
 If $C$ has a counit $\eps_C$ as a coalgebra then $\eps_C\in C^*$ is a unit for the product by definition.

 The dual of a {\em finite} dimensional algebra $A^*=\Hom_R(A,R)$ is a coalgebra with the comultiplication $\Delta(f)(a\ot b):=f(ab)$.
 In general, $\Delta^*$ is a morphism $A^*\to (A\ot A)^*$ and there is an injection $A^*\ot A^*\to ( A \ot A)^*$, which is an isomorphism in the case of a finite dimensional algebra.
Therefore, one restricts to the so-called {\em finite dual} $A^\circ$ which is the maximal subspace of $A^*$ such that $\Delta(A^\circ)\subset A^\circ\ot A^\circ$.
For a Dedekind ring, this is \cite{ChenNichols}:
\begin{gather*}
A^\circ=\{f\in A^*\,|\, \operatorname{Ker}(f) \mbox{ contains an ideal $I$ for which $A/I$ is finitely generated}\}.
\end{gather*}

The Hopf-dual of a Hopf algebra over a field is defined to be $H^\circ$ with the induced structures.
For a graded algebra $A=\bigoplus_{x\in X}A_x$,
we can consider the graded finite dual $\check A^\circ=\bigoplus_{x\in X}X_x^\circ$.

\subsection{Colored coalgebras and algebras}

If $C[X]$ is the setlike coalgebra on $X$,
$C[X]^*$ has multiplication $fg(x)=f(x)g(x)$ and the constant function $u(x)\equiv 1$ is the unit. In~particular, for the dual basis $\delta_x \delta_y=\delta_{x,y}\delta_x$
and $u=\sum_{x\in X}\delta_x$. This is well defined
as the application to any $c\in C[X]$ yields only a finite sum.
Define the algebra $A[X]$ as the $R$ algebra on the monoid $X$ defined by $xy=\delta_{x,y}y$.
This is a~subalgebra of the dual algebra $A[X]\subset C^*[X]$ using the identification $x\leftrightarrow \delta_x$.
Generators are the functions with finite support on $X$.
$A[X]$ is not unital unless
$X$ is finite, but it is almost unital in the sense that $u=\sum_x x\in C[X]^*$ acts as a unit, and when $u$ is restricted to $A[X]$,
it acts locally as a finite sum.

If $X$ is finite, this yields a bialgebra structure on $\Free{X}$.
This is a Hopf algebra if and only if $X$ only has one element, as the putative antipode must satisfy $S(x)=x^{-1}$, but the $x$ act as projectors.

\begin{Lemma}
If $R$ is a Dedekind ring
then
$A[X]^\circ=\finite{A[X]}=C[X]$ under the identification $x\leftrightarrow \delta_x$.
The coproduct is given by decomposition:
\begin{gather*}
\Delta(\delta_x)=\sum_{(x_1,x_2)\colon x_1x_2=x} \delta_{x_1}\ot\delta_{x_2}.
\end{gather*}
\end{Lemma}

\begin{proof}
By \cite{ChenNichols} $f\in A[X]^\circ$ if and only if $A[X]f$ is finitely generated.
Since $xf(y)=f(yx)=\delta_{x,y}f(x)$, we see that $f(x)\delta_x\in A[X]f$. As the $\delta_x$ are a basis, $A[X]f$
being finitely generated is equivalent to almost all $f(x)=0$.
The coproduct on $\finite{A[X]}$
is given by $\D(\delta_x)=\delta_x\ot \delta_x$ and the identification is again by $x\leftrightarrow \delta_x$.
\end{proof}

Linearizing the construction above yields the following definition whose unital version is equivalent to a category enriched in $R$-Mod.
\begin{Definition}
A colored algebra is double graded $R$ module $A=\bigoplus_{(x,y) \in X\times X} A_{x,y}$ with $R$-bilinear
associative multiplications $A_{x,z}\otimes A_{z,y}\to A_{x,y}$.
It is unital if there are elements $\id_x\in A_{x,x}$ which are left and right units under the multiplication.
It is split unital if $A_{x,x}={\rm Rid}_x\oplus \bar A_{x,x}$.
\end{Definition}

For a colored monoid the associated colored algebra is given by $A_{x,y} =\Free{M_{x,y}}$ with the induced multiplication.
The unital structure is the induced structure.

Suppose $M$ is graded, $M=\bigoplus_{x\in X}M_x$, and hence has a right coaction by $C[X]$.
Specializing the general formula for an algebra action induced from a right coalgebra coaction, $f\cdot m =(\id\ot f) (\rho(m))$, the coaction corresponding to
a grading in turn defines an action of $A[X]$ in which the $\delta_x$ act
 by a complete set of orthogonal projectors $\pi_x\colon M\to M_x$, that is $\delta_xm=\pi_x(m)$.
The operation $u=\sum_{x\in X}\delta_x$ acts as unit.
This sum is finite when acting on any $m\in M$ as $M$ is the direct sum of its graded components.

\begin{Lemma}
A colored algebra is an algebra in the category of $A[x]$-bimodules.
\end{Lemma}
\begin{proof}
The multiplication respects the outer grading and taking the product $\ot_{A[X]}$ amounts
to having the inner gradings coincide.
\end{proof}

\begin{Remark}
Note that if one defines a colored algebra as an algebra in $A[x]$-bimodules then one does not obtain finiteness.
Indeed $C^*[X]$ is an $A[X]$ bialgebra. To remedy this one can postulate that $u$ acts locally finitely.
The question that remains is where the finiteness assumptions should be kept and where they could be relaxed.
\end{Remark}

\section{Rota--Baxter operators and algebras}
\label{rbapp}
\begin{Definition}
A Rota--Baxter (RB) operator of weight $\lambda$ on an associative $R$-algebra $A$ a is a linear operator $T\colon A\to A$ that satisfies the equation
\begin{equation}
\label{RBeq}
T(X)T(Y)=T(T(X)(Y))+T(XT(Y))+\lambda T(XY).
\end{equation}
A Rota--Baxter (RB) algebra is a unital associative algebra with a choice of Rota--Baxter ope\-rator.
\end{Definition}

Note that if $T$ is an RB-operator, so is $(1-T)$, furthermore \eqref{RBeq} says that ${\rm Im}(T)$, and similarly ${\rm Im}(1-T)$, are subalgebras of $A$, which will be called
$A_+$ and $A_-$. As $(1-T)+T=1$, we see that $A=A_++A_-$.

{\bf NB:} This sum does not need to be direct.

The pair $(A_+,A_-)$ is called the {\em Birkhoff decomposition} of $A$ defined by $T$.

\begin{Remark}[scaling]
Scaling the RB operator by $\mu$, we see that $\mu T$ is a RB algebra of weight~$\lambda\mu$, so that if $\lambda\in R^*$ is invertible, we can always scale to have $\lambda=0$ or $\lambda = -1$.
\end{Remark}

\begin{Example}
The standard example are Laurent series $k((x))=k[[x]]\big[x^{-1}\big]$, with $T$ being the projection onto the pole part, is an RB-algebra with $\lambda=-1$.
\end{Example}

A generalization of this uses projectors, which are operators with $T^2=1$. In~this case, $1-T$ is also a projector and the two projectors are orthogonal $T(1-T)=0$.
For a projector $T$ on an algebra $A$, set $(1-T)A=A_+$ and $T(A)=A_-$; then $A=A_+\oplus A_-$.

\begin{Proposition}[projectors and subalgebras]
Given a projector $T$, if $A_+$ and $A_-$ are subalgebras, then $T$ is an RB-operator with weight $\lambda=-1$.
Furthermore, $T(1)$ is an idempotent, $T(1)^2=T(1)$,
and multiplication by $T(1)$ is a projector onto a subspace of $A_-$.

 In particular, if $T(1)=0$, then $A_+$ is a unital subalgebra, and,
if $T(1)$ is regular, then $T(1)=1$ and $A_-$ is a unital subalgebra.
\end{Proposition}

\begin{proof}
The conditions means that $(1-T)$ and $T$ are the projectors onto $A=A_+\oplus A_-$.
Since~$A_-$ is a subalgebra, we have that
\begin{equation}
\label{A-eq}
0=(1-T)[T(X)T(Y)]=T(X)T(Y)-T(T(X)T(Y)).
\end{equation}
Since $A_+$ we have that
\begin{gather}
\label{A+eq}
0=T[(1\!-\!T)(X)(1\!-\!T)(Y))]=T(XY)\!-\!T(T(X)Y)\!-\!T(XT(Y)) + T(T(X)T(Y)).
\end{gather}
Plugging \eqref{A-eq} into \eqref{A+eq}, we obtain \eqref{RBeq} for $\lambda=-1$.
Plugging in $X=Y=1$ and using that $T^2(1)=T(1)$ we obtain that $T(1)^2=T(1)$ and hence multiplication by $T(1)$ is a projector.
Plugging in $Y=1$ we obtain $T(1)(1-T)(X)=0$.
\end{proof}

 If $A$ is regular, e.g., a domain, then there are only two possibilities $T(1)=0$ or $T(1)=1$. Thus after possibly switching $T$ and $1-T$,
 we have that $A=A_+\oplus A_-$ with $A_+$ a unital subalgebra and $A_-$ a non-unital subalgebra.
Vice-versa an RB-operator for $\lambda=-1$ comes from a subdirect sum.

\begin{Definition}[\cite{Atkinson}]
Two subalgebras $(A_-,A_+)$ form a subdirect sum of $A$ if
there exists a subalgebra $C\in A\times B$ which satisfies that for
any $a\in A$ there is a unique representation $a=a_-+a_+$, with $(a_-,a_+)\in C$.
\end{Definition}

Such a decomposition defines $T(a)=a_-$ which satisfies \eqref{RBeq} for $\lambda=-1$. Vice-versa given an RB $T$ with $\lambda=-1$, setting $A_-=T(A),A_+=(1-T)(A)$, $C=\{(T(a),(1-T)(a))\colon a\in A\}\subset A_-\times A_+$ gives the subdirect sum.

\begin{Theorem}[\cite{Atkinson}]
The set of RB-operators $T$ with weight $\lambda=-1$ is in bijection with the subdirect sum decomposition of $A$ under the above correspondence.
\end{Theorem}

\section{Set version of graphs}
\label{graphapp}
We give a purely set based, non-graphical version here. A reference for the graphical version is \cite[Section~2.1, Appendix A]{feynman}.
The basic objects will be $\V=\operatorname{Iso}(\FinSet)$.
 This fixes the $\Iso(\V^\ot)$ to have as objects finite tuples of sets $(S_1\kdk S_n)$, where the empty set is allowed both as a tuple and in the
 entries. The morphisms are componentwise or inner isomorphisms
 together with outer permutations, which act as follows:
A permutation $p\in \SS_n$ acts by $p(S_1\kdk S_n)=(S_{p(1)}\kdk S_{p(n)})$.
The inner isomorphisms $\sigma\in \operatorname{Aut}(S_i)$ act as $(\id\kdk \id, \sigma,\id\kdk \id)$. The monoidal product is the joining of tuples.
 In graphical form, these are aggregates, that is a disjoint unions of corollas.
 The outer permutations permute the corollas and the inner isomorphisms the flags of the individual corollas.

There are two versions of the category: the connected and the non-connected version.
 To write the relations in a concise fashion, we
abuse notation and write $\phi$ for $(\id \kdk \id,\phi,\id\kdk \id)$.

For the connected version the morphisms are generated under the monoidal product and concatenation by the isomorphisms and the following morphisms, which called simple edge contractions and simple loop contractions:
\begin{gather*}
\scirct\colon\ (S,T) \to (S\setminus {\s} \amalg T\setminus\{t\}), \qquad
\circ_{ss'}\colon\ (S)\to (S\setminus \{s,s'\}).
\end{gather*}
These satisfy the following relations.
\begin{enumerate}\itemsep=0pt
\item The basic morphisms are equivariant with respect to isomorphisms in the sense that:
 For $\sigma \in \operatorname{Aut}(S)$, $\sigma'\in \operatorname{Aut}(T)$:
\begin{gather*}
\ccirc{\sigma(s)}{\sigma'(t)}(\sigma,\sigma')=\sigma\amalg \sigma'|_{S\setminus \{s\}\amalg T\setminus\{t\}}\scirct, \qquad
\circ_{\sigma(s),\sigma(s')}\sigma=\sigma|_{S\setminus \{s,s'\}}\circ_{s,s'}\!.
\end{gather*}
For an outer permutation:
\begin{gather*}
p\,\circ_{s,s'}=\circ_{s,s'}p, \qquad
p'\ccirc{p(s)}{p(t)}=\scirct p,
\end{gather*}
where $p'$ is the image of the permutation in the embedding $\SS_{n-1}\to \SS_n$ given by the induced block permutation.
\item Loop contractions commute:
 \begin{gather*}
\circ_{s_1,s_1'}\circ_{s_2,s_2'}=\circ_{s_2,s_2'}\circ_{s_1,s_1'}\!.
\end{gather*}
\item Loop and edge contractions commute:
\begin{gather*}
\ccirc{s_1}{t_1} \circ_{s_2s_2'}= \circ_{s_2s_2'}\ccirc{s_1}{t_1}.
\end{gather*}

\item Edge contractions commute if the ``edges'' do not form a cycle.
 For $s_1\in S_1$, $t_1\in T_1$, $s_2\in S_2$, $t_2\in T_2$ and $|\{S_1,T_1,S_2,T_2\}|\geq 3$
\begin{gather*}
\ccirc{s_1}{t_1} \ccirc{s_{2}}{t_2} = \ccirc{s_1}{t_1} \ccirc{s_{2}}{t_2}.
\end{gather*}

\item Edge contractions commute even if they form a cycle, only that after one edge contraction, the second edge contraction is a loop contraction.
For $s_1, s_2\in S$, $t_1, t_2\in T$,
\begin{gather*}
\circ_{s_2,t_2}\ccirc{s_1}{t_1}=\circ_{s_1,t_1}\ccirc{s_2}{t_2}.
\end{gather*}
\end{enumerate}

In the non-connected case, there is an additional generating morphism called
simple merger:
\begin{gather*}
 \mge{S}{T}\colon\ (S,T)\to (S\amalg T).
\end{gather*}

These satisfy the relations:
\begin{enumerate}\itemsep=0pt
\item Mergers are equivariant with respect to the internal action
\begin{gather*}
\mge{S}{T}(\sigma,\id)=(\sigma\amalg \id)\mge{S}{T}.
\end{gather*}
\item Mergers are equivariant with respect to the outer action.
On $(S_1\kdk S_n)$:
\begin{gather*}
\mge{S_{i+1}}{S_{i}} =\mge{S_{i}}{S_{i+1}}\tau_{i,i+1}\qquad \mbox{and}\qquad
p\mge{S_i}{S_{i+1}}=\mge{S_{p'(i)}}{S_{p'(i+1}} p',
\end{gather*}
where $p'$ is the image of the permutation $p$ under the embedding $\SS_{n-1}\to \SS_n$ given by the block permutation with block $(i,i+1)$.

\item Simple mergers commute
\begin{gather*}
\mge{S_1}{T_1}\mge{S_2}{T_2}=\mge{S_2}{T_2}\mge{S_1}{T_1}.
\end{gather*}
\item Mergers and edge/loop contractions on different sets commute.
\item If $s\in S$ and $t\in T$, then
\begin{gather*}
\scirct = \circ_{st}\mge{S}{T}.
\end{gather*}
\end{enumerate}

\subsection*{Acknowledgements}

This paper is the outgrowth of a long lasting conversations with Dirk Kreimer whom we would like to congratulate and profoundly thank for his interest and support. Much of the impetus for this work has come through interactions with him and his group, and it would not have come into existence without Dirk.
RK wishes to thank the Humboldt University, the KMPB, and the Humboldt Foundation for making these visits possible.

We furthermore would like to thank F.~Brown, K.~Yeats and C.~Berger for helpful discussions and the organizers of the
conference {\it Algebraic Structures in Perturbative Quantum Field Theory} at the IHES, in November 2020. The wonderful interactions that took place there led to part of the results of this paper. We also wish to thank the referees for their input which was very beneficial for the paper.

\pdfbookmark[1]{References}{ref}
\LastPageEnding

\end{document}